\documentclass[11pt,fleqn,a4paper]{article}
\usepackage{amsmath,amssymb,amsthm,enumerate}

\newcommand{\p}{{\partial}}

\newtheorem{theorem}{Theorem}

\newtheorem{corollary}[theorem]{Corollary}
\newtheorem{proposition}[theorem]{Proposition}
\newtheorem*{proposition*}{Proposition}
{\theoremstyle{definition}

\newtheorem{remark}[theorem]{Remark}

}

\newcommand{\todo}[1][\null]{\ensuremath{\clubsuit}}

\newcommand{\noprint}[1]{}

\begin{document}

\begin{flushleft}
\LARGE \bf
Equivalence groupoids of classes of linear ordinary differential equations and their group classification
\end{flushleft}

\begin{flushleft}\large
Vyacheslav M.~Boyko$^\dag$, Roman O.~Popovych$^{\dag\ddag}$ and Nataliya M.~Shapoval$^\S$
\end{flushleft}

\noindent $^\dag$~Institute of Mathematics of NAS of Ukraine, 3
Tereshchenkivs'ka Str., 01601 Kyiv-4, Ukraine

\medskip

\noindent
$^\ddag$~Wolfgang Pauli Institut, Universit\"at Wien, Oskar-Morgenstern-Platz 1, A-1090 Wien, Austria

\medskip

\noindent
$^\S$~Faculty of Mechanics and Mathematics, Taras Shevchenko National University of Kyiv,\\
\hphantom{${}^\S$} 4e Academician Glushkov Ave., 03127 Kyiv, Ukraine

\medskip

\noindent
E-mails: boyko@imath.kiev.ua, rop@imath.kiev.ua, natalya.shapoval@gmail.com

\medskip

{\vspace{6mm}\par\noindent\hspace*{5mm}\parbox{150mm}{\small
Admissible point transformations of classes of $r$th order linear ordinary differential equations
(in particular, the whole class of such equations and its subclasses of equations 
in the rational form, the Laguerre--Forsyth form, the first and second Arnold forms)
are exhaustively described.
Using these results, the group classification of such equations is revisited within the algebraic approach in three different ways.
}\par\vspace{6mm}}

\noprint{

\bigskip

}

\section{Introduction}

The study of Lie symmetries of ordinary differential equations (ODEs) has a~long history,
and the ``Lie theory'' was just started as a systematical and elegant approach to integration of various classes of ODEs.
The first results on possible dimensions of the maximal Lie invariance algebras of ODEs
of any fixed order were obtained by Sophus Lie, see, e.g.,~\cite[pp.~294--301]{Lie1893}
and a modern treatment in~\cite[Section~2]{Gonzalez-Gascon&Gonzalez-Lopez1983}.
Namely, Sophus Lie proved that the dimension of the maximal Lie invariance algebra of an $r$th order ODE
is infinite for $r=1$, not greater than~$8$ for $r=2$ and not greater than~$r+4$ for $r\geqslant3$.
He also showed that each ODE of order $r=1$ is similar with respect to a point transformation to the elementary equation $x'=0$
and that, for equations of order $r\geqslant2$,
the maximal dimension of invariance algebras is reached for equations that are reduced by point transformations
to the elementary equation $x^{(r)}=0$;
cf.\ \cite[Theorem~14]{Olver1994} and \cite[Theorems~6.39 and~6.43]{Olver1995}.

In spite of the fact that transformational and, in particular, symmetry properties of linear ODEs were intensively investigated
(see, e.g., detailed reviews \cite{Ibragimov1992,Mahomed2007,Schwarz2000}
and textbooks \cite{Berkovich2002,Ibragimov1999,Olver1995,Schwarz2008}),
in the present paper we consider them from another side,
describing the whole set of admissible transformations between such equations.
This creates a basis for the group classification of linear ODEs within the framework of the algebraic approach,
which is quite effective for solving group classification problems
for both ordinary and partial differential equations;
see, e.g., \cite{Bihlo2012,Popovych&Bihlo2012,Popovych&Kuzinger&Eshraghi2010,Popovych2008,Vaneeva&Popovych2009} and references therein.
Previously, in \cite{Krause&Michel,Mahomed&Leach1990}, the group classification of linear ODEs
was carried out within the framework of the standard ``compatibility'' approach 
based on the study of compatibility of the determining equations for Lie symmetries and the direct solution of these equations, 
which led to cumbersome calculations.
Although the ``compatibility'' approach is the most commonly used in group analysis of differential equations, 
it is efficient only for classes of simple structure.
\mbox{In~\cite[pp.~217--218]{Olver1995}} the solution of the group classification problem of linear ODEs
was related to Wilczynski's result~\cite{Wilczynski1906} on relative invariants of the Laguerre--Forsyth form of these equations.
The similar problem on the classification of linear ODEs up to contact transformations
as well as the associated equivalence problem were considered in detail in~\cite{Yumaguzhin2000a,Yumaguzhin2000b,Yumaguzhin2000c}.

The main purpose of the present paper is to carry out the complete group classification
of the class~$\mathcal L$ of $r$th order ($r\geqslant3$) linear ODEs in more elegant algebraic ways,
using subalgebra analysis of the equivalence algebra associated with~$\mathcal L$.
This properly works since the class~$\mathcal L$ is (pointwise) normalized (in the usual sense),
i.e., transformations from its (usual) point equivalence (pseudo)group%
\footnote{There exist other names for this notion, e.g., ``structure invariance group'' \cite{Schwarz2008}.
The attribute ``usual'' and the prefix ``pseudo-'' are usually omitted for usual equivalence pseudogroups.
We will also say ``normalization'' without attributes in the case of pointwise normalization in the usual sense.
\vspace{1ex}
}%
~$G^\sim$
generate all admissible point transformations%
\footnote{%
An \emph{admissible (point) transformation} of a class of differential equations is a triple of the form
$(\mathcal E,\tilde {\mathcal E},\mathcal T)$.
Here $\mathcal E$ and $\tilde {\mathcal E}$ are equations from the class or,
equivalently, the corresponding values of the arbitrary elements parameterizing the class.
The transformational part~$\mathcal T$ of the admissible transformation
is a point transformation mapping the equation~$\mathcal E$ to the equation~$\tilde {\mathcal E}$.\vspace{1ex}
}
between equations from~$\mathcal L$.
The group classification of the class of second-order linear ODEs is trivial
since its equivalence group acts transitively.
Note that the equivalence group~$G^\sim$ of the class~$\mathcal L$ with $r\geqslant2$ was first found
by St\"ackel~\cite{Stackel1893}.%
\footnote{%
Recall also the contribution by
Halphen~\cite{Halphen1878}, Laguerre~\cite{Laguerre1879}, Forsyth~\cite{Forsyth1888} and Wilczynski~\cite{Wilczynski1906}
in the study of point transformations between linear ODEs.%
}
The set of admissible transformations of any class of differential equations possesses the groupoid structure
and is called the \emph{equivalence groupoid} of this class~\cite{Bihlo2012,Popovych&Bihlo2012}.
See, e.g., \cite{Bihlo2012,Popovych&Bihlo2012,Popovych&Kuzinger&Eshraghi2010,Vaneeva&Popovych&Sophocleous2013}
for the definition of normalized classes and other related notions.
So, we can say that the equivalence groupoid~$\mathcal G^\sim$ of the class~$\mathcal L$ with $r\geqslant3$
is generated by its equivalence group~$G^\sim$.

In Section~\ref{Equivalence_groupoids} we begin the study of the class~$\mathcal L$ of $r$th order $(r\geqslant2)$ linear ODEs
with the description of its equivalence groupoid in terms of its equivalence group and normalization.
One can gauge arbitrary elements of the class~$\mathcal L$ by parameterized families of transformations
from~$G^\sim$, which induces mappings of the class~$\mathcal L$ onto its subclasses.
Two such gauges for arbitrary elements related to the subleading-order derivatives are well known.
They result in the rational form with the first subleading coefficient being equal to zero
(the subclass~$\mathcal L_1$)
and Laguerre--Forsyth form with the first two subleading coefficients being equal to zero
(the subclass~$\mathcal L_2$).
It appears that for $r\geqslant3$ both the subclasses~$\mathcal L_1$ and~$\mathcal L_2$ are also normalized
with respect to their equivalence groups.
Then we study two gauges for arbitrary elements related to the lowest-order derivatives, which give the first and second Arnold forms.
The corresponding subclasses are even not semi-normalized
and hence these gauges are not convenient for symmetry analysis.
Having the chain of nested normalized classes~$\mathcal L\supset\mathcal L_1\supset\mathcal L_2$ for $r\geqslant3$ 
and the associated chain of classes of homogeneous equations, which are peculiarly semi-normalized,
we can classify Lie symmetries of $r$th order linear ODEs
within the algebraic approach in three different ways, which is done in Section~\ref{Group_classification}.
Section~\ref{SectionOnGenExtEquivGroupsOfClassesOfLinearODEs} is devoted to 
the study of the generalized extended equivalence groups of the above classes of linear ODEs 
and to the improvement of normalization properties of these classes by reparameterization. 
In the final section we summarize results of the paper and discuss their connection with possible approaches
to solving group classification problems for classes of systems of linear differential equations.

\section{Equivalence groupoids of classes of linear ODEs}\label{Equivalence_groupoids}

Consider the class~$\mathcal L$ of $r$th order ($r\geqslant 2$) linear ODEs, which have the form
\begin{equation}\label{ODE}
x^{(r)}+a_{r-1}(t)x^{(r-1)}+\dots+a_1(t)x^{(1)}+a_0(t)x=b(t),
\end{equation}
where $a_{r-1}$, $\dots$, $a_1$, $a_0$ and $b$ are arbitrary smooth functions of~$t$,
$x=x(t)$ is the unknown function, $x^{(k)}=d^kx/dt^k$, $k=1,\dots,r$, and $x^{(0)}:=x$.
Below we also use the notation $x'=dx/dt$ and $x''=d^2x/dt^2$
for the first and second derivatives, respectively.
The subscripts~$t$ and~$x$ denote differentiation with respect to the corresponding variables.
We assume that all variables, functions and other values are either real or complex, i.e.,
the underlying field $\mathbb F$ is either $\mathbb F=\mathbb R$ or $\mathbb F=\mathbb C$, respectively.
We work within the local approach.

\subsection{General class}

There are two different cases for the structure of the equivalence groupoid of the entire class~$\mathcal L$
depending on the value $r$, namely $r=2$ and $r\geqslant3$.
We begin with the case $r=2$.

\begin{proposition}\label{proposition1}
The equivalence group $G^\sim$ of the class~$\mathcal L$ with $r=2$ consists of the transformations
whose projections to the variable space%
\footnote{%
There is no nontrivial gauge equivalence in all the classes of linear ODEs considered in this section 
(which is not the case for reparameterized classes studied in Section~\ref{SectionOnGenExtEquivGroupsOfClassesOfLinearODEs}). 
This is why for all equivalence transformations in Section~\ref{Equivalence_groupoids} 
we present only the components corresponding to the variables~$t$ and~$x$
since each pair of these components completely determines the corresponding transformation components for arbitrary elements.
As the form~\eqref{trans1} is the most general for all relevant equivalence transformations,
the transformation components for arbitrary elements can be derived using Fa\`a di Bruno's formula and the general Leibniz rule,
and thus they are quite cumbersome.\looseness=-1\vspace{-1ex}
}
have the form
\begin{gather}\label{trans1}
\tilde t=T(t), \quad \tilde x=X_1(t)x+X_0(t),
\end{gather}
where $T$, $X_1$ and $X_0$ are arbitrary smooth functions of~$t$ with $T_tX_1\ne0$.
\end{proposition}

\begin{proof}
Suppose that a point transformation $\mathcal T$ of the general form
\begin{gather}\label{general_trans}
\tilde t=T(t,x), \quad \tilde x=X(t,x),
\end{gather}
where the Jacobian $J=|\p(T,X)/\p(t,x)|$ does not vanish, $J\ne0$,
connects two fixed second-order linear ODEs~$\mathcal E$ and~$\mathcal {\tilde E}$.
We substitute the expressions for the new variables (which are with tildes) and the corresponding derivatives
in terms of the old variables (which are without tildes) into $\mathcal {\tilde E}$.
The equality obtained should be identically satisfied on solutions of the equation~$\mathcal E$.
Therefore, additionally substituting the expression for $x''$ implied by~$\mathcal E$,
we can split the equality with respect to the derivative $x'$.
Collecting the coefficients of~$(x')^3$, we obtain the equation
\[
X_{xx}T_x-X_xT_{xx}+\tilde a_1X_xT_x^{\,2}+(\tilde a_0X-\tilde b)T_x^{\,3}=0.
\]
As~$\tilde a_1$, $\tilde a_0$ and~$\tilde b$ are the only arbitrary elements involved in this equation
and we study the equivalence group, we can vary the arbitrary elements and hence split with respect to them.
Hence $T_x=0$, i.e., $T=T(t)$. 
Then the terms with~$(x')^2$ give the equation $T_tX_{xx}=0$.
As the condition $J\ne0$ reduces to the inequality $T_tX_x\ne0$, we have $X=X_1(t)x+X_0(t)$ with $X_1\ne0$.
The other determining equations, which are derived by the additional splitting with respect to $x'$ and~$x$,
define the transformation components for the arbitrary elements
as functions of the variables and the arbitrary elements.
\end{proof}

\begin{proposition}\label{proposition2}
The equivalence groupoid of the class~$\mathcal L$ of second-order linear ODEs is generated by 
compositions of transformations from the equivalence group~$G^\sim$ of this class 
with transformations from the point symmetry group of the equation $x''=0$.
Therefore, the class~$\mathcal L$ is semi-normalized but not normalized.
\end{proposition}

\begin{proof}
The free particle equation $x''=0$ admits point symmetry transformations 
that are truly fractional linear with respect to~$x$ or whose components for~$t$ depend on~$x$. 
Each of these properties is not consistent with the transformation form~\eqref{trans1}. 
This is why there exists admissible transformations in the class~$\mathcal L$ 
that are not generated by its equivalence transformations, 
i.e., this class is not normalized.

It is commonly known that any second-order linear ODE~$\mathcal E$ is locally reduced to the equation $x''=0$
by an equivalence transformation, so-called Arnold transformation,
\begin{gather}\label{EqArnoldTrans}
\tilde t=\frac{\varphi_2(t)}{\varphi_1(t)}, \quad \tilde x=\frac{x-\varphi_0(t)}{\varphi_1(t)},
\end{gather}
where $\varphi_0$ is a particular solution of the equation $\mathcal E$, $\varphi_1$ and $\varphi_2$
are linearly independent solutions of the corresponding homogeneous equation, see, e.g.,~\cite[p.~43]{Arnold1988} or \cite{Gonzalez-Lopez1988}.
In other words, the class~$\mathcal L$ is a~single orbit of its equivalence group~$G^\sim$.
Any class with this property is semi-normalized.
We show this in detail.
Consider two fixed equations $\mathcal E_1$ and $\mathcal E_2$ from the class~$\mathcal L$ with $r=2$ and
a point transformation~$\mathcal T$ linking these equations.
Let $\mathcal T_1$ and $\mathcal T_2$ be the projections of elements of $G^\sim$ that map the
equations~$\mathcal E_1$ and~$\mathcal E_2$, respectively, to $x''=0$.
Then the transformation $\mathcal T_0:=\mathcal T_2\mathcal T\mathcal T_1^{-1}$ belongs to the point symmetry group of the equation $x''=0$.
This implies the representation $\mathcal T=\mathcal T^{-1}_2\mathcal T_0\mathcal T_1$.
Roughly speaking, $\mathcal T$ is the composition of the equivalence transformations~$\mathcal T_1$ and~$\mathcal T_2^{-1}$
and the symmetry transformation~$\mathcal T_0$ of the equation $x''=0$.
It is obvious that any transformation possessing such a representation maps the equation~$\mathcal E_1$ to the equation~$\mathcal E_2$.
The above representation can be rewritten as $\mathcal T=\check{\mathcal T}\hat{\mathcal T}$, where $\check{\mathcal T}=\mathcal T_2^{-1}\mathcal T_1$
is an equivalence transformation of the class~$\mathcal L$ and $\hat{\mathcal T}=\mathcal T_1^{-1}\mathcal T_0\mathcal T_1$
is a symmetry transformation of the equation~$\mathcal E_1$, which means that the class~$\mathcal L$ is semi-normalized.
\end{proof}

\begin{remark}\label{RemarkOnFiber-PreservingAdmTransOf2ndOrderLODEs}
The equivalence group $G^\sim$ induces all fiber-preserving admissible transformations of the class~$\mathcal L$ with $r=2$.
This directly follows from the fact that imposing the constraint~$T_x=0$ for admissible transformations also implies the condition~$X_{xx}=0$.
\end{remark}

In what follows we consider the class~$\mathcal L$ with $r\geqslant3$.
Although the projections of transformations from the equivalence group to the variable space
in the case~$r\geqslant3$ coincide with that in the case~$r=2$,
the corresponding equivalence groupoids have different structures.

\begin{proposition}\label{proposition3}
The equivalence group $G^\sim$ of the class~$\mathcal L$, where $r\geqslant3$, consists of the transformations
whose projections to the variable space have the form~\eqref{trans1}.
This group generates the entire equivalence groupoid of the class~$\mathcal L$, i.e., the class~$\mathcal L$ is normalized.
\end{proposition}

\begin{proof}
In order to study admissible transformations in the class~$\mathcal L$, we consider a pair of equations from this class,
namely an equation~$\mathcal E$ of the form~\eqref{ODE}
and an equation~$\tilde{\mathcal E}$ of the same form, where all variables, derivatives and arbitrary elements are with tildes,
\noprint{
\begin{equation}\label{tildeODE}
\tilde x^{(r)}+\tilde{a}_{r-1}(\tilde t)\tilde x^{(r-1)}+\dots+\tilde{a}_1(\tilde t)\tilde x'+\tilde{a}_0(\tilde t)\tilde x=\tilde{b}(\tilde t),
\end{equation}
}%
and assume that these equations are connected by a point transformation $\mathcal T$ of the general form~\eqref{general_trans}.
At first we express derivatives with tildes in terms of the variables without tildes,
\begin{gather*}
\tilde x^{(k)}=\left(\frac{1}{DT}D\right)^kX,
\end{gather*}
where $D=\p_t+x'\p_x+x'' \p_{x'}+\cdots$ is the total derivative operator with respect to the variable~$t$.
After substituting the expressions for the variables and derivatives with tildes into~$\tilde{\mathcal E}$,
we derive an equation in the variables without tildes.
It should be an identity on the manifold determined by~$\mathcal E$ 
in the $r$th order jet space with the independent variable~$t$ and the dependent variable~$x$.
The coefficient of $x''x^{(r-1)}$ in this equation equals
\[
-\frac{J}{(DT)^{r+2}}T_x\left(3+\frac{(r-2)(r+3)}2\right)=0,
\]
and hence $T_x=0$, i.e., the function $T$ does not depend on the variable $x$, $T=T(t)$.
The nondegeneracy condition $J\ne0$ is simplified to $T_tX_x\ne0$.
Taking into account the condition $T_x=0$, we collect coefficients of $x'x^{(r-1)}$, which gives $rT_t^{-r}X_{xx}=0$.
This equation implies that $X_{xx}=0$, i.e., $X$ is an affine function of~$x$, $X=X_1(t)x+X_0(t)$.
Therefore, the transformation~$\mathcal T$ has the form~\eqref{trans1}.
The other determining equations, which are obtained by splitting with respect to derivatives of $x$
after substituting for~$x^{(r)}$ in view of~$\mathcal E$,
establish the relation between arbitrary elements of the initial and the transformed equations.

The transformation~$\mathcal T$ maps any equation from the class~$\mathcal L$ to another equation from the same class,
and its prolongation to the arbitrary elements $a_{r-1}$,~$\dots$, $a_0$ and~$b$,
which is given by the above relation,
is a point transformation in the joint space of the variables and the arbitrary elements.
Hence such prolongations of the transformations of the form~\eqref{trans1} constitute 
the (usual) equivalence group~$G^\sim$ of the class~$\mathcal L$.
Since any admissible transformation in the class~$\mathcal L$ is induced by a transformation from~$G^\sim$, 
this class is normalized.
\end{proof}

Consider the corresponding subclass~$\widehat{\mathcal L}$ of $r$th order ($r\geqslant2$) linear homogeneous ODEs,
which is singled out from the class~$\mathcal L$ by the constraint~$b=0$.
The arbitrary element~$b$ can be gauge to zero by equivalence transformations.
Namely, the class~$\mathcal L$ is mapped to its subclass~$\widehat{\mathcal L}$ by a family of point transformations with
$T=t$, $X^1=1$ and $X^0$ being a particular solution of the initial equation,
and thus these transformations are parameterized by~$b$.

\begin{corollary}\label{CorollaryOnEquivGroupOfLinHomogenODEs}
The equivalence group~$\widehat G^\sim$ of the subclass~$\widehat{\mathcal L}$
is obtained from the equivalence group~$G^\sim$ of the class~$\mathcal L$
by setting $X_0=0$ and neglecting the transformation component for~$b$.
\end{corollary}

\begin{corollary}\label{CorollaryOnEquivGroupoidOfLinHomogenODEsOfOrder2}
The subclass~$\widehat{\mathcal L}$ with $r=2$ is a single orbit of the free particle equation $x''=0$
under the action of the equivalence group~$\widehat G^\sim$.
Hence the subclass~$\widehat{\mathcal L}$ is semi-normalized but not normalized.
\end{corollary}

\begin{corollary}\label{CorollaryOnEquivGroupoidOfLinHomogenODEsOfOrderGreaterThan2}
Given any equation~$\mathcal E$ from the subclass~$\widehat{\mathcal L}$ with $r\geqslant3$,
a point transformation maps~$\mathcal E$ to another equation from the same subclass
if and only this transformation has the form~\eqref{trans1}, where the ratio~$X_0/X_1$ is a solution of~$\mathcal E$.
\end{corollary}

In other words, the transformational part of any admissible transformation
within the class~$\widehat{\mathcal L}$ with $r\geqslant3$ can be represented
as the composition of a linear superposition symmetry transformation of the initial equation
with the projection of an element of the equivalence group~$\widehat G^\sim$ to the variable space.
For all equations from~$\widehat{\mathcal L}$, the associated groups of linear superposition symmetry transformations
are of the same structure. In particular, they are commutative and $r$-dimensional.
Therefore, although the class of~$\widehat{\mathcal L}$ is not normalized, it is semi-normalized in a quite specific way,
which is a particular case of so-called \emph{uniform semi-normalization}.%
\footnote{%
Similar properties are known for classes of homogeneous linear PDEs
whose corresponding classes of (in general, inhomogeneous) linear PDEs are normalized~\cite{Popovych2008}.
}
For short, in similar situations we will say that a class is uniformly semi-normalized with respect to linear superposition of solutions.

\subsection{Rational form}

Using parameterized families of transformations from the equivalence group $G^\sim$, we can gauge arbitrary elements of the class~$\mathcal L$.
For example, we can set $a_{r-1}=0$. This gauge can be realized by the parameterized family of projections of equivalence transformations to the $(t,x)$-space
\begin{gather}\label{EqTransToRationalForm}
\tilde t=t, \quad \tilde x=\exp\left(\frac{1}{r}\int a_{r-1}(t)dt\right)x,
\end{gather}
which maps the class~$\mathcal L$ onto the subclass $\mathcal L_1$ of equations in the \textit{rational form}
\begin{equation}\label{ODE1}
x^{(r)}+a_{r-2}(t)x^{(r-2)}+\dots+a_1(t)x'+a_0(t)x=b(t),
\end{equation}
where we omitted tildes over the variables and the arbitrary elements.
This form was used in \cite{Krause&Michel,Mahomed&Leach1990}
for the group classification of linear ODEs within the framework of the standard ``compatibility'' approach.

\begin{proposition}\label{proposition4}
The equivalence group~$G^\sim_1$ of the subclass~$\mathcal L_1$
consists of the transformations whose projections to the variable space have the form
\begin{gather}\label{trans2}
\tilde t=T(t), \quad \tilde x=C(T_t(t))^{\frac{r-1}2}x+X_0(t),
\end{gather}
where $T$ and $X_0$ are arbitrary smooth functions of~$t$ with $T_t\ne0$,%
\footnote{%
For even~$r$, the power of~$T_t$ is half-integer
and hence the absolute value of~$T_t$ should be substituted instead of~$T_t$ in the real case
or a branch of square root should be fixed in the complex case.\label{FootnoteOnPowersOfTt}
}
and $C$ is an arbitrary nonzero constant.
The subclass~$\mathcal L_1$ with $r=2$ is semi-normalized.
If $r\geqslant3$, then the group~$G^\sim_1$ generates the equivalence groupoid of this subclass, i.e., the subclass is normalized.
\end{proposition}

\begin{proof}
In the case $r=2$ we follow the proof of Proposition~\ref{proposition1}
and derive the form~\eqref{trans1} for equivalence transformations of the subclass~$\mathcal L_1$.
Then further collecting coefficients of~$x'$ and~$x$ gives the equation
\begin{gather*}
\frac{X_1}{T_t^{\,2}}\left(\frac{X_{1,t}}{X_1}-\frac12\frac{T_{tt}}{T_t}\right)=0,
\end{gather*}
which is integrated to the relation $\smash{X_1=CT_t{}^{\frac12}}$ with an arbitrary nonzero constant~$C$,
as well as the equivalence transformation components for the arbitrary elements~$a_0$ and~$b$
of the subclass~$\mathcal L_1$ with $r=2$.
The semi-normalization of this subclass is proved in the same way as Proposition~\ref{proposition2}.

In the case $r\geqslant3$ we describe the entire equivalence groupoid.
Suppose that an equation~$\mathcal E$ of the form~\eqref{ODE1}
and an equation~$\tilde{\mathcal E}$ of the same form, where all variables, derivatives and arbitrary elements are with tildes,
are connected by a point transformation~$\mathcal T$.
In view of Proposition~\ref{proposition3} this transformation has the form~\eqref{trans1}.
We express the variables with tildes and the corresponding derivatives in terms of the variables and derivatives without tildes,
substitute the obtained expressions into~$\tilde{\mathcal E}$
and collect terms containing the derivative~$x^{(r-1)}$, which gives
\[
r\frac{X_1}{T_t^{\,r}}\left(\frac{X_{1,t}}{X_1}-\frac{r-1}2\frac{T_{tt}}{T_t}\right)x^{(r-1)}.
\]
The coefficient of~$x^{(r-1)}$ vanishes only if $\smash{X_1=CT_t{}^{\frac{r-1}2}}$, where $C$ is an arbitrary nonzero constant.
Therefore, the point transformation $\mathcal T$ has the form~\eqref{trans2}.
Analogously to the end of the proof of Proposition~\ref{proposition3},
the transformations of this form when prolonged to the arbitrary elements $a_{r-2}$,~$\dots$, $a_0$ and~$b$
constitute the equivalence group $G^\sim_1$ of the subclass~$\mathcal L_1$.
\end{proof}

The status and properties of the corresponding subclass~$\widehat{\mathcal L}_1$ of homogeneous equations within the class~$\mathcal L_1$
are the same as those of the subclass~$\widehat{\mathcal L}$ within the class~$\mathcal L$.

\begin{corollary}\label{CorollaryOnEquivGroupOfLinHomogenODEsInRationalForm}
The equivalence group~$\widehat G^\sim_1$ of the subclass~$\widehat{\mathcal L}_1$
is derived from the equivalence group~$G^\sim_1$ of the class~$\mathcal L_1$
by setting $X_0=0$ and neglecting the transformation component for~$b$.
\end{corollary}

The equivalence groupoid of the subclass~$\widehat{\mathcal L}_1$ is exhaustively described 
by the following two assertions depending on the value of~$r$. 

\begin{corollary}\label{CorollaryOnEquivGroupoidOfLinHomogenODEsInRationalFormOfOrder2}
If $r=2$, then the subclass~$\widehat{\mathcal L}_1$ is semi-normalized but not normalized
since it is a single orbit of the free particle equation $x''=0$ under the action of the group~$\widehat G^\sim_1$.
\end{corollary}

\begin{corollary}\label{CorollaryOnEquivGroupoidOfLinHomogenODEsInRationalFormOfOrderGreaterThan2}
For each equation~$\mathcal E$ from the subclass~$\widehat{\mathcal L}_1$ with $r\geqslant3$,
a point transformation is the transformational part of an admissible transformation in~$\widehat{\mathcal L}_1$ with~$\mathcal E$ as source 
if and only if it is of the form~\eqref{trans2},
where the product~$T_t{}^{-\frac{r-1}2}X_0$ is a solution of~$\mathcal E$.
This means that this subclass is uniformly semi-normalized with respect to linear superposition of solutions.
\end{corollary}

\subsection{Laguerre--Forsyth form}

Transformations from $G^\sim_1$ are parameterized by an arbitrary function $T=T(t)$ with $T_t\ne0$.
Hence we can set $a_{r-2}=0$ in the equation~\eqref{ODE1} by a transformation from the group $G^\sim_1$,
where the parameter-function $T$ is a solution of the equation
\begin{gather*}
T_{ttt}T_t-\frac32T_{tt}^{\,\,2}+\frac{12}{r(r^2-1)}a_{r-2}T_t^{\,4}=0. 
\end{gather*}
Thus, a family of such equivalence transformations parameterized by the arbitrary element~$a_{r-2}$
maps the subclass~$\mathcal L_1$ onto the narrower subclass~$\mathcal L_2$ of equations in the \textit{Laguerre--Forsyth form}
\begin{equation}\label{ODE2}
x^{(r)}+a_{r-3}(t)x^{(r-3)}+\dots+a_1(t)x'+a_0(t)x=b(t).
\end{equation}
Note that, in contrast to the transformation~\eqref{EqTransToRationalForm},
the above map does not preserve the corresponding subclass of linear ODEs with constant coefficients.

\begin{proposition}\label{proposition5}
The equivalence group~$G^\sim_2$ of the subclass~$\mathcal L_2$ with $r\geqslant2$ consists of the transformations
whose projections to the variable space have the form
\begin{gather}
\tilde t=\frac{\alpha t+\beta}{\gamma t+\delta}, \quad \tilde x=\frac C{(\gamma t+\delta)^{r-1}}x+X_0(t), \label{trans3}
\end{gather}
where $\alpha$, $\beta$, $\gamma$, $\delta$ and~$C$ are arbitrary constants with $\alpha\delta-\beta\gamma\ne0$ and $C\ne0$
that are defined up to obvious rescaling (so, only four constants among them are essential),
and $X_0$ is an arbitrary smooth function of~$t$.
In the case $r=2$ the subclass~$\mathcal L_2$ is semi-normalized but not normalized.
If $r\geqslant3$, then the group~$G^\sim_2$ generates the equivalence groupoid of this subclass, i.e., the subclass is normalized.
\end{proposition}

\begin{proof}
We should again consider the cases $r=2$ and $r\geqslant3$ separately.

For $r=2$ we repeat the proof of Proposition~\ref{proposition1} and the corresponding part of the proof of Proposition~\ref{proposition4}
and derive the form~\eqref{trans1} for equivalence transformations of the subclass~$\mathcal L_2$.
Then further collecting coefficients of~$x$ gives the equation
\begin{gather}\label{EqWithSchwarzianDer}
\frac{T_{ttt}}{T_t}-\frac32\left(\frac{T_{tt}}{T_t}\right)^2=0.
\end{gather}
In other words, the Schwarzian derivative of the function~$T$ vanishes, i.e.,
the function $T$ is fractional linear,
\begin{gather}\label{EqTisFractionalLinear}
T(t)=\frac{\alpha t+\beta}{\gamma t+\delta},
\end{gather}
where $\alpha$, $\beta$, $\gamma$, $\delta$ are arbitrary constants with $\alpha\delta-\beta\gamma\ne0$
that are defined up to nonvanishing constant multiplier.
The rest of terms in the determining equation results in the equivalence transformation component for the arbitrary element~$b$
of the subclass~$\mathcal L_2$ with $r=2$.
The semi-normalization of this subclass is proved in the same way as Proposition~\ref{proposition2}.

In order to describe the entire equivalence groupoid the subclass~$\mathcal L_2$ in the case $r\geqslant3$,
we suppose that a point transformation~$\mathcal T$ links equations~$\mathcal E$ and~$\tilde{\mathcal E}$ from the class~$\mathcal L_2$.
(We assume that all the values in the equation~$\tilde{\mathcal E}$ are with tildes.)
Thus,~$\mathcal T$ has the form~\eqref{trans2}.
We express the derivatives of $\tilde x$ with respect to~$\tilde t$ in terms of~$(t,x)$,
substitute these expressions into the equation~$\tilde{\mathcal E}$
and then substitute the expression for~$x^{(r)}$ implied by~$\mathcal E$.
Collecting the coefficients of $x^{(r-2)}$ gives the equation~\eqref{EqWithSchwarzianDer}, i.e.,
the function $T$ is of the form~\eqref{EqTisFractionalLinear}.
We substitute the expression for~$T$ into~\eqref{trans2} and obtain transformations that map any equation
from the subclass~$\mathcal L_2$ to another equation from the same subclass,
and the new arbitrary elements functionally depend on the old variables and the old arbitrary elements.
Therefore, prolongations of these transformations to the arbitrary elements $a_{r-3}$,~$\dots$, $a_0$ and~$b$
constitute the equivalence group~$G^\sim_2$ of the subclass~$\mathcal L_2$.
\end{proof}

Denote by~$\widehat{\mathcal L}_2$ the subclass of homogeneous equations within the class~$\mathcal L_2$.
The case $r=2$ is singular for transformational properties of~$\widehat{\mathcal L}_2$
since then the only element of this subclass is the free particle equation $x''=0$.
This implies that the subclass~$\widehat{\mathcal L}_2$ is normalized,
and its equivalence group coincides with the point symmetry group of the equation $x''=0$,
which consists of fractional linear transformations in the space of~$(t,x)$.
The case $r\geqslant3$ for~$\widehat{\mathcal L}_2$ is similar to ones for~$\widehat{\mathcal L}$ and~$\widehat{\mathcal L}_1$.

\begin{corollary}\label{CorollaryOnEquivGroupAndEquivGroupoidOfLinHomogenODEsInLaguerreForsythFormOfOrderGreaterThan2}
If $r\geqslant3$,
the equivalence group~$\widehat G^\sim_2$ of the subclass~$\widehat{\mathcal L}_2$
is derived from the equivalence group~$G^\sim_2$ of the class~$\mathcal L_2$
by setting $X_0=0$ and neglecting the transformation component for~$b$.
The subclass~\smash{$\widehat{\mathcal L}_2$} is uniformly semi-normalized with respect to linear super\-position of solutions.
More precisely, the equivalence groupoid of this subclass
is constituted by triples of the form~$(\mathcal E,\tilde{\mathcal E},\mathcal T)$, 
where the equation-source~$\mathcal E$ runs through the entire subclass~$\widehat{\mathcal L}_2$, 
the transformational part~$\mathcal T$ is of the form~\eqref{trans3}
with the product~$(\gamma t+\delta)^{r-1}X_0$ being an arbitrary solution of~$\mathcal E$, 
and the equation-target is defined by $\tilde{\mathcal E}=\mathcal T(\mathcal E)$.
\end{corollary}

\begin{remark}
In fact, the equivalence and admissible transformations for the classes~$\widehat{\mathcal L}$, $\widehat{\mathcal L}_1$ and $\widehat{\mathcal L}_2$
were known in the literature for a long time due to St\"ackel, Laguerre, Forsyth et al.;
see, e.g., \cite[Section~4.1]{Schwarz2008} and \cite[Chapter~I and \S~II.4]{Wilczynski1906}. 
At the same time, we rigorously formulate these results and explicitly describe the associated equivalence groupoids by proving
that these classes are normalized if $r\geqslant3$ and semi-normalized if $r=2$.
\end{remark}

\subsection{First Arnold form}

The above gauges of arbitrary elements of the class~$\mathcal L$, which are commonly used,
concern the subleading coefficients~$a_{r-1}$ and~$a_{r-2}$, but this is not a unique possibility.
Instead of~$a_{r-1}$ and~$a_{r-2}$ one can gauge the lowest coefficients $a_0$ and $a_1$.
We can set $a_0=0$ in any equation from the class~$\mathcal L$ by an Arnold transformation
\begin{gather*}
\tilde t=t, \quad \tilde x=\dfrac{x}{\varphi_1(t)},
\end{gather*}
where $\varphi_1$ is a nonzero solution of the corresponding homogeneous equation.
As a result, we obtain the subclass~$\mathcal A_1$ of the class~$\mathcal L$ that is constituted by the equations of the form
\begin{equation}\label{1st_Arnold's_form}
x^{(r)}+a_{r-1}(t)x^{(r-1)}+\dots+a_1(t)x^{(1)}=b(t).
\end{equation}
Following the (eponymous) Arnold Principle,%
\footnote{%
The Arnold Principle states that if a notion bears a personal name, then this name is not the name of the discoverer.
The Berry Principle extends the Arnold Principle by stating the following: the Arnold Principle is applicable to itself.}
we call this form the \emph{first Arnold form}.

\begin{proposition}\label{proposition6}
The equivalence groupoid of the class~$\mathcal A_1$, where $r\geqslant3$, is constituted by
the admissible transformations whose equations-sources exhaust the whole class~$\mathcal A_1$ and 
whose transformational parts are of the form
\begin{gather}\label{1stArnoldAdmTrans}
\tilde t=T(t), \quad \tilde x=\frac{x}{\psi_1(t)}+X_0(t),
\end{gather}
where $T$ and $X_0$ are arbitrary smooth functions of~$t$ with $T_t\ne0$,
and $\psi_1=\psi_1(t)$ is a nonzero solution of the homogeneous equation associated with the corresponding equation-source.
\end{proposition}

\begin{proof}
Let~$\mathcal T$ be a point transformation between equations~$\mathcal E$ and~$\tilde{\mathcal E}$ of the form~\eqref{1st_Arnold's_form}.
Then the transformation~$\mathcal T$ is of the general form~\eqref{trans1}.
As the equation~$\tilde{\mathcal E}$ belongs to the class~$\mathcal A_1$,
the function $\smash{\tilde{\psi_1}\equiv1}$ is a solution of the corresponding homogeneous equation.
Hence $\psi_1=1/X_1$ is a solution of the homogeneous equation associated with~$\mathcal E$.
Therefore, the transformation~$\mathcal T$ is of the form~\eqref{1stArnoldAdmTrans}.
\end{proof}

\begin{remark}\label{RemarkOnFiber-PreservingAdmTransOf2ndOrder1stArnoldForm}
In the case $r=2$, the same assertion is true for the fiber-preserving subgroupoid of the equivalence groupoid of the class~$\mathcal A_1$;
cf.\ Remark~\ref{RemarkOnFiber-PreservingAdmTransOf2ndOrderLODEs}.
\end{remark}

\begin{proposition}\label{proposition7}
The equivalence group~$G^\sim_{\mathcal A_1}$ of the class~$\mathcal A_1$, where $r\geqslant2$,
consists of the transformations whose projections to the variable space have the form
\begin{gather}\label{1stArnoldEquivTrans}
\tilde t=T(t), \quad \tilde x=Cx+X_0(t),
\end{gather}
where $T$ and $X_0$ are arbitrary smooth functions of~$t$ with $T_t\ne0$ and $C$ is an arbitrary nonzero constant.
\end{proposition}

\begin{proof}
The projection of any transformation from the group~$G^\sim_{\mathcal A_1}$ to the variable space is of the form~\eqref{1stArnoldAdmTrans}.
For $r\geqslant3$, this is an obvious consequence of Proposition~\ref{proposition6}.
In the case $r=2$, similarly to the proof of Proposition~\ref{proposition1},
we can first show that projections of equivalence transformations to the variable space are fiber-preserving
and then take into account Remark~\ref{RemarkOnFiber-PreservingAdmTransOf2ndOrder1stArnoldForm}.

Note that constant functions are solutions of any homogeneous equation from the class~$\mathcal A_1$.
Therefore, the group~$G^\sim_{\mathcal A_1}$ contains the transformations
whose restrictions on the space of $(t,x)$ have the form~\eqref{1stArnoldEquivTrans}.
Moreover, only these transformations are in the group~$G^\sim_{\mathcal A_1}$.
Indeed, consider a transformation~$\mathcal T$ of the form~\eqref{1stArnoldAdmTrans} with $\psi_1\ne{\rm const}$.
Then there exists a homogeneous equation from~$\mathcal A_1$ that is not satisfied by the function $\psi_1$.
Hence the coefficient $\tilde a_0$ of the corresponding transformed equation is nonzero.
This means that the transformation~$\mathcal T$ is not a projection of an element of the group~$G^\sim_{\mathcal A_1}$.
\end{proof}

By \smash{$\widehat{\mathcal A}_1$} we denote the subclass consisting of homogeneous equations of the form~\eqref{1st_Arnold's_form},
i.e.,\ singled out from the class~$\mathcal A_1$  by the constraint $b=0$.

\begin{corollary}\label{Corollary1On1stArnoldForm}
The equivalence group~$\widehat G^\sim_{\mathcal A_1}$ of the subclass~$\widehat{\mathcal A}_1$ with $r\geqslant2$ is obtained
from the group~$\smash{G^\sim_{\mathcal A_1}}$ by setting $X_0=0$ and neglecting the transformation component for the arbitrary element~$b$.
A~point transformation~$\mathcal T$ relates two equations from this subclass
if and only if it has the form~\eqref{1stArnoldAdmTrans},
where the product~$\psi_1X_0$ is a solution of the corresponding initial equation.
\end{corollary}

\begin{corollary}\label{Corollary2On1stArnoldForm}
The classes~$\mathcal A_1$ and~$\widehat{\mathcal A}_1$ with $r\geqslant3$ are not semi-normalized.
\end{corollary}

\begin{proof}
There exists an equation~$\mathcal E$ in~$\widehat{\mathcal A}_1\subset\mathcal A_1$
whose point symmetry group consists only of transformations related to the linearity and the homogeneity of~$\mathcal E$.%
\footnote{%
A common fact is that similar equations have similar point symmetry groups. 
Moreover, any admissible transformation in the class~$\widehat{\mathcal L}$ maps 
point symmetries associated to the linearity and homogeneity of the corresponding initial equation  
to symmetries of the same kind that are admitted by the target equation.
Hence it suffices to find an equation~$\widetilde{\mathcal E}$ with trivial point symmetries in the class~$\widehat{\mathcal L}_2$.
In view of Corollary~\ref{CorollaryOnEquivGroupAndEquivGroupoidOfLinHomogenODEsInLaguerreForsythFormOfOrderGreaterThan2},
possible point transformations of~$\widetilde{\mathcal E}$ within the class~$\widehat{\mathcal L}_2$ are exhausted,
up to point symmetries associated to the linearity and the homogeneity of~$\widetilde{\mathcal E}$, 
by the transformations of the form~\eqref{trans3}, where $C=1$ and $X^0=0$. 
Equations that are not invariant with respect to any of such transformations exist even among
equations from~\smash{$\widehat{\mathcal L}_2$} with polynomial coefficients.
}
These transformations are of the form $\tilde t=t$, $\tilde x=\widehat Cx+\widehat X^0(t)$,
where $\widehat C$ is an arbitrary nonzero constant and $\widehat X_0$ is an arbitrary solution of~$\mathcal E$.
The compositions of symmetry transformations of~$\mathcal E$ and projections of transformations
from~$G^\sim_{\mathcal A_1}$ (resp.\ \smash{$\widehat G^\sim_{\mathcal A_1}$}) to the variable space
are at most of the form~\eqref{1stArnoldEquivTrans}.
At the same time, the equation~$\mathcal E$ has nonconstant solution.
This implies that some admissible transformations of~$\mathcal E$, which are of the form~\eqref{1stArnoldAdmTrans},
are not generated by the above compositions.
\looseness=-1
\end{proof}

\begin{remark}
Analogously to Proposition~\ref{proposition2},
the classes~$\mathcal A_1$ and~$\widehat{\mathcal A}_1$ with $r=2$ are semi-normalized but not normalized 
since they are orbits of the free particle equation $x''=0$
with respect to the equivalence groups~$G^\sim_{\mathcal A_1}$ and~$\widehat G^\sim_{\mathcal A_1}$, respectively.
\end{remark}

\subsection{Second Arnold form}

Using a transformation of the form~\eqref{EqArnoldTrans} with $\varphi_0=0$,
we can set additionally $a_1=0$ in any equation from the class~$\mathcal A_1$.
In this way the class~$\mathcal A_1$ (and thus the entire class~$\mathcal L$)
is mapped onto its subclass~$\mathcal A_2$ that consists of the equations of the form
\begin{equation}\label{2st_Arnold's_form}
x^{(r)}+a_{r-1}(t)x^{(r-1)}+\dots+a_2(t)x^{(2)}=b(t),
\end{equation}
called the \emph{second Arnold form}.

\begin{proposition}\label{proposition8}
The  equivalence groupoid of the subclass~$\mathcal A_2$, where $r\geqslant3$, is constituted by
the admissible transformations  whose equations-sources exhaust the whole class~$\mathcal A_2$ and 
whose transformational parts are of the form
\begin{gather}\label{2ndArnoldAdmTrans}
\tilde t=\frac{\psi_2(t)}{\psi_1(t)}, \quad \tilde x=\frac{x}{\psi_1(t)}+X_0(t),
\end{gather}
where $\psi_1=\psi_1(t)$ and $\psi_2=\psi_2(t)$ are arbitrary linearly independent solutions of the homogeneous equation 
associated with the corresponding equation-source, and $X_0$ is an arbitrary smooth function of~$t$.
\end{proposition}

\begin{proof}
Suppose that a point transformation~$\mathcal T$ connects two equations~$\mathcal E$ and~$\tilde{\mathcal E}$ from the class~$\mathcal A_2$.
Then the transformation~$\mathcal T$ has the form~\eqref{1stArnoldAdmTrans}.
Note that the function $\tilde\psi_2=\tilde t$ is a solution of the homogeneous equation associated with~$\tilde{\mathcal E}$.
Hence the function $\psi_2=\psi_1T$ is a solution of the homogeneous equation corresponding to~$\mathcal E$, i.e., $T=\psi_2/\psi_1$.
As a result, the transformation~$\mathcal T$ is of the form~\eqref{2ndArnoldAdmTrans}.
\end{proof}

\begin{proposition}\label{proposition9}
The equivalence group~$G^\sim_{\mathcal A_2}$ of the subclass~$\mathcal A_2$, where $r\geqslant3$,
consists of the transformations whose projections to the variable space have the form
\begin{gather}\label{2ndArnoldEquivTrans}
\tilde t=\frac{\alpha t+\beta}{\gamma t+\delta}, \quad \tilde x=\frac{x}{\gamma t+\delta}+X_0(t),
\end{gather}
where $\alpha$, $\beta$, $\gamma$ and $\delta$ are arbitrary constants with $\alpha\delta-\beta\gamma\ne0$ 
and $X_0$ is arbitrary smooth functions of~$t$.
\end{proposition}

\begin{proof}
Any transformation from the group~$G^\sim_{\mathcal A_2}$ generates a family of elements from the equivalence groupoid of the subclass~$\mathcal A_2$,
and hence its projection to the variable space has the form~\eqref{2ndArnoldAdmTrans}.
Since all affine functions of~$t$ are solutions of any homogeneous equation from the subclass~$\mathcal A_2$,
then the group~$G^\sim_{\mathcal A_2}$ contains
the transformations whose restrictions on the variable space are of the form~\eqref{2ndArnoldEquivTrans}.
We prove that there are no other transformations in the group~$G^\sim_{\mathcal A_2}$, i.e.,\
a~transformation~$\mathcal T$ of the form~\eqref{2ndArnoldAdmTrans} does not correspond to an equivalence transformation of the class~$\mathcal A_2$
if at least one of the parameter-functions $\psi_1$ or $\psi_2$ is a nonlinear function of~$t$.

First consider the case where $\psi_1$ is nonlinear.
We take a homogeneous equation $\mathcal E$ from the subclass~$\mathcal A_2$ that is not satisfied by~$\psi_1$.
Then the corresponding transformed equation $\tilde{\mathcal E}$ possesses no nonvanishing constant solutions and hence its coefficient $\tilde a_0$ is nonzero.
This means that the equation~$\tilde{\mathcal E}$ does not belong to the subclass~$\mathcal A_2$, which gives the necessary statement.

Now consider the complementary case, namely, where the function $\psi_1$ is affine and the function $\psi_2$ is nonlinear.
Then there is a homogeneous equation~$\mathcal E$ of the form~\eqref{2st_Arnold's_form} that is not satisfied by~$\psi_2$.
As~$\psi_1$ is a solution of the equation~$\mathcal E$, 
constant functions satisfy the corresponding transformed equation~$\tilde{\mathcal E}$
and hence its coefficient~$\tilde a_0$ is equal to zero.
Moreover, the function~$\tilde\psi_2\equiv\tilde t$ is not a solution of the equation~$\tilde{\mathcal E}$,
so its coefficient $\tilde a_1$ is nonzero.
Therefore, the equation~$\tilde{\mathcal E}$ is not contained in the subclass~$\mathcal A_2$, which completes the proof.
\end{proof}

Transformational properties of the subclass~$\widehat{\mathcal A}_2$ of homogeneous equations from~$\mathcal A_2$ with $r\geqslant3$
are similar to ones of the class~$\widehat{\mathcal A}_1$.

\begin{corollary}\label{Corollary1On2ndArnoldForm}
The equivalence group~$\widehat G^\sim_{\mathcal A_2}$ of the subclass~$\widehat{\mathcal A}_2$ with $r\geqslant3$ is obtained
from the group~$\smash{G^\sim_{\mathcal A_2}}$ by setting $X_0=0$ and neglecting the transformation component for the arbitrary element~$b$.
A~point transformation relates two equations from this subclass
if and only if it has the form~\eqref{2ndArnoldAdmTrans},
where the product~$\psi_1X_0$ is a solution of the corresponding initial equation.
\end{corollary}

\begin{corollary}\label{Corollary2On2ndArnoldForm}
The classes~$\mathcal A_2$ and~$\widehat{\mathcal A}_2$ with $r\geqslant3$ are not semi-normalized.
\end{corollary}

\begin{proof}
Following the proof of Corollary~\ref{Corollary2On1stArnoldForm},
consider an admissible transformation $(\mathcal E_1,\mathcal E_2,\mathcal T)$ in the class~\smash{$\widehat{\mathcal A}_2\subset\mathcal A_2$},
where the point symmetry group of~$\mathcal E_1$  consists only of the transformations related to the linearity and the homogeneity of~$\mathcal E_1$,
in the representation~\eqref{2ndArnoldAdmTrans} for~$\mathcal T$ the solution~$\psi_1$ of~$\mathcal E_1$ is not affine in~$t$, $\psi_1''\ne0$,
and $\mathcal E_2=\mathcal T(\mathcal E_1)$.
The compositions of symmetry transformations of~$\mathcal E_1$ with projections of transformations
from~$G^\sim_{\mathcal A_2}$ (resp.\ \smash{$\widehat G^\sim_{\mathcal A_2}$}) to the variable space
are at most of the form~\eqref{2ndArnoldEquivTrans}.
Therefore, the admissible transformation $(\mathcal E_1,\mathcal E_2,\mathcal T)$ is not generated by one of the above compositions.
\end{proof}

\begin{remark}
For $r=2$, the classes~$\mathcal A_2$ and~$\widehat{\mathcal A}_2$ coincide
with the classes~$\mathcal L_2$ and~$\widehat{\mathcal L}_2$, respectively.
\end{remark}

\section{Group classification}\label{Group_classification}

As remarked earlier, any second-order linear ODE can be reduced to the free particle equation $x''=0$ (see, e.g.,~\cite{Lie1888}),
which admits the eight-dimensional Lie invariance algebra
\begin{gather*}
\langle \p_t, \ \p_x, \ t\p_t, \ x\p_t, \ t\p_x, \ x\p_x, \ tx\p_t+x^2\p_x, \ t^2\p_t+tx\p_x\rangle.
\end{gather*}
This gives the exhaustive group classification of second-order linear ODEs.

Let $r\geqslant3$. It is also a classical result by Sophus Lie~\cite[pp.~296--298]{Lie1893} that the dimension of
Lie invariance algebras for $r$th order ODEs with $r\geqslant3$ is not greater than $r+4$.
Much later this result was partially reproved in \cite{Krause&Michel,Mahomed&Leach1990} only for linear ODEs.

Any inhomogeneous linear ODE can be reduced to the corresponding homogeneous equation
by a transformation from the equivalence group $G^\sim$.
The class~$\mathcal L$ with $r\geqslant3$ is normalized 
and its subclass~$\widehat{\mathcal L}$ of homogeneous equations is uniformly semi-normalized with respect to linear superposition of solutions.
Therefore, for solving the group classification problem for the class~$\mathcal L$
it suffices to solve the similar problem for the subclass~$\widehat{\mathcal L}$.

Consider an $r$th order linear homogeneous ODE~$\mathcal E$.
Due to linearity leading to the linear superposition principle,
this equation admits the $r$-dimensional abelian Lie invariance
algebra~$\mathfrak g_{\rm a}^{\mathcal E}$ spanned by the vector fields
\begin{gather}\label{fund set of solution}
\varphi_1(t)\p_x, \ \varphi_2(t)\p_x, \ \dots, \ \varphi_r(t)\p_x ,
\end{gather}
where the functions $\varphi_i=\varphi_i(t)$, $i=1,\dots,r$, form a fundamental set of solutions of the equation~$\mathcal E$.
By virtue of homogeneity, the equation~$\mathcal E$ also admits
the one-parameter symmetry group of scale transformations generated by the vector field $x\p_x$.
Therefore, each equation~$\mathcal E$ from the subclass~$\widehat{\mathcal L}$ admits the $(r{+}1)$-dimensional Lie invariance algebra
\begin{gather}\label{(r+1)-dim}
\mathfrak g_0^{\mathcal E}= \langle x\p_x, \ \varphi_1(t)\p_x, \ \varphi_2(t)\p_x, \ \dots, \ \varphi_r(t)\p_x \rangle.
\end{gather}
Corollary~\ref{CorollaryOnEquivGroupoidOfLinHomogenODEsOfOrderGreaterThan2} implies that
any Lie symmetry operator $Q$ of the equation~$\mathcal E$ is of the general form $Q=\tau(t)\p_t+(\xi_1(t)x+\xi_0(t))\p_x$,
where $\tau$, $\xi_1$ and $\xi_0$ are smooth functions of~$t$,
and $\xi_0$ is additionally a solution of~$\mathcal E$.
Hence the maximal Lie invariance algebra~$\mathfrak g^{\mathcal E}$ of~$\mathcal E$ contains the subalgebra~$\mathfrak g_0^{\mathcal E}$ as an ideal. 
This is why the group classification of the class~$\mathcal L$ means 
the classification of the quotient algebras $\mathfrak g^{\mathcal E}/\mathfrak g_0^{\mathcal E}$, where~$\mathcal E$ runs through~$\mathcal L$, 
up to  $G^\sim$-equivalence.

Now we carry out the group classification of linear ODEs in three different ways, which respectively involve
the Laguerre--Forsyth form~\eqref{ODE2}, the rational form~\eqref{ODE1} and
Lie's classification of realizations of finite-dimensional Lie algebras by vector fields in the space of two variables.

\subsection{The first way: Laguerre--Forsyth form}

The group classification of the class~$\widehat{\mathcal L}$ can be further reduced
to the group classification of its subclass~$\widehat{\mathcal L}_2$
of $r$th order homogeneous linear ODEs in the Laguerre--Forsyth form,
which is singled out from~$\widehat{\mathcal L}$ by the constraint $a_{r-1}=a_{r-2}=0$.
Indeed, both the arbitrary elements $a_{r-1}$ and $a_{r-2}$ can be gauged to zero by a family of point transformations,
which are parameterized by these arbitrary elements and are associated with equivalence transformations.
Moreover, equations from the class~$\widehat{\mathcal L}$ are $\widehat G^\sim$-equivalent if and only if
their images in the class~$\widehat{\mathcal L}_2$ are $\widehat G^\sim_2$-equivalent.

In view of Corollary~\ref{CorollaryOnEquivGroupAndEquivGroupoidOfLinHomogenODEsInLaguerreForsythFormOfOrderGreaterThan2},
the class~\smash{$\widehat{\mathcal L}_2$} is not normalized.
At the same time, it is uniformly semi-normalized with respect to linear superposition of solutions
(i.e., with respect to the symmetry groups of its equations related to linear superposition of their solutions).
This suffices for using an advanced version of the algebraic method to group classification of the class~$\widehat{\mathcal L}_2$.
The equivalence group~$\widehat G^\sim_2$ of~$\widehat{\mathcal L}_2$ consists of the transformations
whose projections to the variable space are of the form~\eqref{trans3} with $X_0=0$.
The scalings of~$x$ constitute the kernel group~$\widehat G^\cap_2$ of~$\widehat{\mathcal L}_2$
since they are only common point symmetry transformations for all equations from the class~$\widehat{\mathcal L}_2$.
By trivial prolongation to the arbitrary elements,
the group~$\widehat G^\cap_2$ is embedded into~$\widehat G^\sim_2$ as a normal subgroup.
The factor group $\widehat G^\sim_2/\widehat G^\cap_2$ can by identified with the subgroup~$H$ of~$\widehat G^\sim_2$
singled out by the constraint $C=1$, where additionally the expression $\alpha\delta-\beta\gamma$ is equal to 1 or $\pm1$ 
in the complex ($\mathbb F=\mathbb C$) or the real ($\mathbb F=\mathbb R$) case, respectively.
Up to point symmetry transformations associated with the linear superposition principle and homogeneity, i.e., related to the Lie algebra~$\mathfrak g_0^{\mathcal E}$,
all admissible transformations of the equation~$\mathcal E$ within~$\widehat{\mathcal L}_2$ are exhausted by the projections of elements from~$H$ to the variable space.
The subgroup~$H$ is isomorphic to the projective general linear group~$\mathrm{PGL}(2,\mathbb F)$ of fractional linear transformations of~$t$.

Therefore, all possible Lie symmetry extensions within the class~$\widehat{\mathcal L}_2$
are necessarily associated with subgroups of~$\mathrm{PGL}(2,\mathbb F)$.
In infinitesimal terms, the maximal Lie invariance algebra of the equation~$\mathcal E$ is a semidirect sum of the algebra~$\mathfrak g_0^{\mathcal E}$
and a subalgebra of the realization of~$\mathfrak{sl}(2,\mathbb F)$ spanned by the vector fields
\[
\mathcal P=\p_t,\quad
\mathcal D=t\p_t+\tfrac 12 (r-1)x\p_x,\quad
\mathcal K=t^2\p_t+(r-1)tx\p_x.
\]
Subalgebras of~$\mathfrak{sl}(2,\mathbb F)$ are well known (see, e.g.,~\cite{Patera1977a} or the appendix in the arXiv version of~\cite{Popovych&Boyko&Nesterenko&Lutfullin2003}).
Up to internal automorphisms of~$\mathfrak{sl}(2,\mathbb F)$,
a complete list of subalgebras 
of~$\mathfrak{sl}(2,\mathbb F)$ is exhausted by the zero subalgebra $\{0\}$,
the one-dimensional subalgebras $\langle\mathcal P\rangle$, $\langle\mathcal D\rangle$ and, only for $\mathbb F=\mathbb R$, $\langle\mathcal P+\mathcal K\rangle$,%
\footnote{%
The subalgebra $\langle\mathcal P+\mathcal K\rangle$ is equivalent to $\langle\mathcal P\rangle$ if $\mathbb F=\mathbb C$.
This is why the part of the consideration corresponding to the subalgebra $\langle\mathcal P+\mathcal K\rangle$ can be then just neglected in the complex case.
\vspace{.5ex}
}
the two-dimensional subalgebra $\langle \mathcal P, \, \mathcal D\rangle$
and the entire realization $\langle \mathcal P, \, \mathcal D, \, \mathcal K\rangle$.
The equivalence of subalgebras of~$\mathfrak{sl}(2,\mathbb F)$
well agrees with the similarity of equations within the class~$\widehat{\mathcal L}_2$.

The zero subalgebra $\{0\}$ corresponds to the general case with no extension.

For one-dimensional extensions of algebras of the form~\eqref{(r+1)-dim} by the subalgebras $\langle\mathcal P\rangle$,
$\langle\mathcal D\rangle$ and $\langle\mathcal P+\mathcal K\rangle$,
the corresponding equations from the class~$\widehat{\mathcal L}_2$ respectively take the forms
\begin{gather}
x^{(r)}+c_{r-3}x^{(r-3)}+\dots+c_1 x'+c_0 x=0, \label{linear-constant-equation} \\
x^{(r)}+c_{r-3}t^{-3}x^{(r-3)}+\dots+c_1 t^{-r+1} x'+c_0 t^{-r}x=0,\label{linear-Euler-equation}\\
x^{(r)}+q_{r-3}(t)x^{(r-3)}+\dots+q_1(t) x'+q_0 (t) x =0, \label{linear-projective-equation}
\end{gather}
where $c_0, \dots,c_{r-3}$ are arbitrary constants,
\begin{gather*}
q_{r-3}(t)=\frac{c_{r-3}}{(1+t^2)^3}, \\
q_{m}(t)=\frac{c_m}{(1+t^2)^{r-m}}
-\frac{(m+1)(r{-}m{-}1)}{(1+t^2)^{r-m}}\int \big(1+t^2\big)^{r-m-1}q_{m+1}(t)dt, \quad m=r{-}4,\dots,0,
\end{gather*}
and the integral denotes a fixed antiderivative.
However, by the point transformations%
\footnote{%
For the first transformation, the absolute value of~$t$ should be substituted instead of~$t$ in the real case
or branches of the ln and, if $r$ is even, power functions should be fixed in the complex case.
}
\begin{gather*}
\tilde t=\ln t, \ \tilde x =xt^{-(r-1)/2} \quad \text{and} \quad  \tilde t=\arctan t, \ \tilde x =x\big(1+t^2\big)^{-(r-1)/2}
 \end{gather*}
the maximal Lie invariance algebras of equations of the form~\eqref{linear-Euler-equation}
(known as the Euler--Cauchy equation, or just Euler's equation)
and the form~\eqref{linear-projective-equation} are reduced to the algebras looking as
\begin{gather}\label{(r+2)-dim}
\langle \p_{\tilde t}, \ \tilde x\p_{\tilde x}, \
\tilde\varphi_1(\tilde t)\p_{\tilde x}, \ \tilde\varphi_2(\tilde t)\p_{\tilde x}, \ \dots, \ \tilde\varphi_r(\tilde t)\p_{\tilde x} \rangle.
\end{gather}
Moreover, the above transformations map these equations to constant-coefficient equations from the class~\eqref{ODE1},
where $a_{r-2}=-\tfrac{1}{24}r(r^2-1)$ for~\eqref{linear-Euler-equation} and $a_{r-2}=\tfrac 16r(r^2-1)$ for~\eqref{linear-projective-equation}.
Additionally scaling~$t$, we can set $a_{r-2}=-1$ and $a_{r-2}=1$, respectively.
Thus, any equation from~$\widehat{\mathcal L}_2$ that admits an $(r{+}2)$-dimensional Lie invariance algebra
is equivalent to a~homogeneous equation with constant coefficients from the class~\eqref{ODE1},
in which $a_{r-2}=0$, $a_{r-2}=-1$ and $a_{r-2}=1$ for~\eqref{linear-constant-equation},
\eqref{linear-Euler-equation} and \eqref{linear-projective-equation}, respectively.

If an equation from~$\widehat{\mathcal L}_2$ possesses the $(r{+}3)$-dimensional Lie invariance algebra
\begin{gather*}
\langle \p_t, \ t\p_t+\tfrac 12 (r-1)x\p_x, \ x\p_x, \ \varphi_1(t)\p_x, \ \varphi_2(t)\p_x, \ \dots, \ \varphi_r(t)\p_x \rangle,
\end{gather*}
then it has the form $x^{(r)}=0$ and hence its maximal Lie invariance algebra is $(r{+}4)$-dimensional,
\begin{gather}\label{(r+4)-dim}
\langle \p_t, \ t\p_t+\tfrac 12 (r-1)x\p_x, \ t^2\p_t+(r-1)tx\p_x, \ x\p_x, \ \varphi_1(t)\p_x, \ \varphi_2(t)\p_x, \ \dots, \ \varphi_r(t)\p_x \rangle.
\end{gather}
As the functions $\varphi_1$, \dots, $\varphi_r$ form a fundamental set of solutions of the elementary equation $x^{(r)}=0$,
we can choose $\varphi_i=t^{i-1}$, $i=1,\dots,r$.
Thus, there is no $r$th order linear ODE whose maximal Lie invariance algebra is $(r{+}3)$-dimensional.
Moreover, if such an equation admits an $(r{+}4)$-dimensional Lie invariance algebra,
then it is similar to the elementary equation $x^{(r)}=0$ with respect to a point transformation of the form~\eqref{trans1}.

\subsection{The second way: rational form}

Consider now the subclass~$\widehat{\mathcal L}_1$ of homogeneous ODEs of the rational form~\eqref{ODE1}.
Again, the group classification of the class~$\widehat{\mathcal L}$ reduces to that of the subclass~$\widehat{\mathcal L}_1$
since the gauge $a_{r-1}=0$ singling out the subclass~$\widehat{\mathcal L}_1$ from~$\widehat{\mathcal L}$
is realized by a family of equivalence transformations. 
Equations from the class~$\widehat{\mathcal L}$ are $\widehat G^\sim$-equivalent if and only if
their images in the class~$\widehat{\mathcal L}_1$ are $\widehat G^\sim_1$-equivalent. 
Moreover, the subclass~$\widehat{\mathcal L}_1$ is uniformly semi-normalized with respect to linear superposition of solutions.

Each equation~$\mathcal E$ from~$\widehat{\mathcal L}_1$ possesses
the $(r{+}1)$-dimensional Lie invariance algebra~$\mathfrak g_0^{\mathcal E}$,
which is an ideal of the maximal Lie invariance algebra~$\mathfrak g^{\mathcal E}$ of~$\mathcal E$.
At the same time, in contrast to the Laguerre--Forsyth form,
the equivalence group~$\widehat G_1^\sim$ of~$\widehat{\mathcal L}_1$ is parameterized by an arbitrary function~$T=T(t)$.
Corollary~\ref{CorollaryOnEquivGroupoidOfLinHomogenODEsInRationalFormOfOrderGreaterThan2} implies that
the algebra~$\mathfrak g^{\mathcal E}$ is contained in the algebra $\langle \mathcal R(\tau)\rangle+\mathfrak g_0^{\mathcal E}$,
where plus denotes the sum of vector spaces,
\[\mathcal R(\tau)=\tau(t)\p_t+\tfrac 12(r-1)\tau_t(t)x\p_x,\]
and the parameter~$\tau$ runs through the set of smooth functions of~$t$.
It is easy to see that 
\[[\mathcal R(\tau^1),\mathcal R(\tau^2)]=\mathcal R(\tau^1\tau_t^2-\tau^2\tau^1_t).\]
Hence $\mathfrak g_1^{\mathcal E}:=\langle \mathcal R(\tau)\rangle\cap\mathfrak g^{\mathcal E}$ is
a (finite-dimensional) subalgebra of~$\mathfrak g^{\mathcal E}$.
Each vector field $\mathcal R(\tau)$ is completely defined by its projection $\mathrm{pr}_t\mathcal R(\tau)$
to the space of the variable~$t$, $\mathrm{pr}_t\mathcal R(\tau)=\tau(t)\p_t$.
In other words, the algebras~$\mathfrak g_1^{\mathcal E}$ and $\mathrm{pr}_t\mathfrak g_1^{\mathcal E}$ are isomorphic.
Moreover, the corresponding projection of the equivalence group of the subclass~$\widehat{\mathcal L}_1$
coincides with the group of all local diffeomorphisms in the space of the variable~$t$.

As a result, the group classification of the class~$\widehat{\mathcal L}_1$ reduces to the classification
of (local) realizations of finite-dimensional Lie algebras by vector fields in the space of the single variable~$t$.
The latter classification is well known and was done by Sophus Lie himself.
A complete list of inequivalent realizations on the line is exhausted by the algebras
$\{0\}$, $\langle\p_t\rangle$, $\langle\p_t, \, t\p_t\rangle$ and $\langle\p_t, \, t\p_t, \, t^2\p_t\rangle$,
which gives $\{0\}$, $\langle \mathcal P\rangle$, $\langle \mathcal P,\,\mathcal D\rangle$ and $\langle \mathcal P,\,\mathcal D,\,\mathcal K\rangle$
as possible inequivalent Lie symmetry extensions within the class~$\widehat{\mathcal L}_1$.
Here $\mathcal P=\mathcal R(1)$, $\mathcal D=\mathcal R(t)$ and $\mathcal K=\mathcal R(t^2)$ are the same operators as those in the first way.
If the equation~$\mathcal E$ admits the two-dimensional extension $\langle \mathcal P,\,\mathcal D\rangle$,
then it coincides with the elementary equation $x^{(r)}=0$,
which possesses the three-dimensional extension $\langle \mathcal P,\,\mathcal D,\,\mathcal K\rangle$.
This is why the two-dimensional extension is improper.
Finally, we have three inequivalent cases of Lie symmetry extensions in the class~$\mathcal L$, which are
\begin{itemize}\itemsep=0.5ex
\item
the general case with no extension,
\item
general constant-coefficient equations admitting the one-dimensional extension $\langle \mathcal P\rangle$, and
\item
the elementary equation $x^{(r)}=0$ possessing the three-dimensional extension $\langle \mathcal P,\,\mathcal D,\,\mathcal K\rangle$.
\end{itemize}

\subsection{The third way: general form}

The group classification of the entire class~$\widehat{\mathcal L}$ can also be obtained directly from Lie's classification of
realizations of finite-dimensional Lie algebras  by vector fields in the spaces of two real or complex variables~\cite{Lie1888,LieTransformationsgruppen}.
A modern treatment of these results on realizations was presented, e.g., in~\cite{Gonzalez-Lopez1992,Nesterenko2006,Olver1994,Olver1995}.
In order to solve the group classification problem for the class~$\widehat{\mathcal L}$,
from Lie's list of realizations we select candidates for the maximal invariance algebras of equations from~$\widehat{\mathcal L}$.
All candidates should satisfy the following obvious properties, which are preserved by point transformations:
\begin{itemize}\itemsep=0.5ex
\item
The maximal Lie invariance algebra~$\mathfrak g^{\mathcal E}$ of each $r$th order linear ODE~$\mathcal E$ $(r\geqslant3)$
contains the $(r{+}1)$-dimensional almost abelian ideal~$\mathfrak g_0^{\mathcal E}$.
More precisely, the ideal~$\mathfrak g_0^{\mathcal E}$ is the semidirect sum
of an $r$-dimensional abelian ideal~$\mathfrak g_{\rm a}^{\mathcal E}$ of the whole algebra~$\mathfrak g^{\mathcal E}$
and the linear span of one more vector field whose adjoint action on the ideal~$\mathfrak g_{\rm a}^{\mathcal E}$ is the identity operator.
\item
Moreover, the ideal~$\mathfrak g_0^{\mathcal E}$ is an intransitive Lie algebra of vector fields, $\mathop{\rm rank}\mathfrak g_0^{\mathcal E}=1$.
\end{itemize}
The above properties are satisfied by
realization families 21, 23, 26 and 28 from \protect{\cite[Table~1]{Gonzalez-Lopez1992}} (or
realization families 3.2, 1.6, 1.9 and 1.11 from \protect{\cite[pp 472--473]{Olver1995}}, or
realization families 49, 51, 54 and 56 from \protect{\cite[Table~1.1]{Nesterenko2006}},
respectively).
As in the previous two ways, among $r$th order linear ODEs, only the elementary equation~$x^{(r)}=0$ admits the third realization.
At the same time, this equation also possesses the fourth realization, which is of greater dimension than the third one.
This is why the third realization should be neglected.

The equivalence within the chosen families of realizations well conforms with the point equivalence of linear ODEs.
Indeed, given two such realizations that are equivalent with respect to a point transformation~$\mathcal T$
and whose $(r{+}1)$-dimensional almost abelian ideals are of the form~\eqref{(r+1)-dim},
the transformation~$\mathcal T$ should have the form~\eqref{trans1},
where $X_0/X_1$ is a linear combinations of the parameters-functions $\varphi_1$, \dots, $\varphi_r$ of the initial realization.
The proof of this claim is based on two facts.
The first fact is trivial:
The mapping generated by a point transformation between realizations of a Lie algebra by vector fields
is a Lie algebra isomorphism and, in particular, establishes a bijection between the corresponding nilradicals.
The second fact is that nilradicals of appropriate realizations
are spanned by vector fields $\varphi_1(t)\p_x$, \dots, $\varphi_r(t)\p_x$,
where $\varphi_i=\varphi_i(t)$, $i=1,\dots,r$, are linearly independent functions.
This fact is obvious for realizations reducing to the forms~\eqref{(r+1)-dim} and~\eqref{(r+4)-dim}.
Suppose that it is not the case for a realization reducing to the form~\eqref{(r+2)-dim},
where tildes over all values are omitted.
Then for some constant~$\nu$ the corresponding nilradical includes the vector field $\p_t+\nu x\p_x$,
and commutation relations within the nilradical imply
the existence of a constant nilpotent matrix~$(\mu_{ij})$
such that \[\varphi_i'-\nu\varphi_i=\mu_{i1}\varphi_1+\dots+\mu_{ir}\varphi_r.\]
The constant~$\nu$ can be set to zero by a point transformation $\bar t=t$, $\bar x=e^{\nu t}x$.
As the functions~$\varphi_i$ are linearly independent,
up to their linear combining the matrix~$(\mu_{ij})$ can be assumed to coincide with
the $r\times r$ nilpotent Jordan block, and hence we can set~$\varphi_i=t^{i-1}$.
At the same time, the only equation that belongs to the class~$\widehat{\mathcal L}$
and is invariant with respect to the algebra $\langle\p_x,\,t\p_x,\dots,t^{r-1}\p_x\rangle$
is the free particle equation~$x''=0$,
but the maximal Lie invariance algebra of this equation is of higher dimension.

\medskip\par
As a result, using the algebraic method of group classification 
we have reproved the following assertion in three different ways 
(see, e.g.,~\cite{Krause&Michel,Mahomed&Leach1990,Olver1995,Schwarz2008,Yumaguzhin2000c}):

\begin{proposition}\label{proposition10}
The dimension of the maximal Lie invariance algebra~$\mathfrak g^{\mathcal E}$
of an $r$th order ($r\geqslant3$) linear ODE~$\mathcal E$ takes a value from $\{r+1,r+2,r+4\}$. 
In the general case $\dim\mathfrak g^{\mathcal E}=r+1$ 
the algebra~$\mathfrak g^{\mathcal E}$ is exhausted by the Lie symmetries 
that are associated with the linearity of~$\mathcal E$. 
If~$\dim\mathfrak g^{\mathcal E}\geqslant r+2$, then the equation~$\mathcal E$
is similar to a linear ODE with constant coefficients.
In~the case $\dim\mathfrak g^{\mathcal E}=r+4$ the equation~$\mathcal E$
is reduced by a point transformation of the form~\eqref{trans1}
to the elementary equation $x^{(r)}=0$.
\end{proposition}

\section{Generalized extended equivalence groups}\label{SectionOnGenExtEquivGroupsOfClassesOfLinearODEs}

For each normalized class like~$\mathcal L$, $\mathcal L_1$ and $\mathcal L_2$ with $r\geqslant 3$, 
its equivalence groupoid is generated by its (usual point) equivalence group. 
The question is whether it is possible to generalize the notion of equivalence groups 
in such a way that other classes of linear ODEs will become normalized in this generalized sense. 

Given a class of (systems of) differential equations, 
every element in its usual equivalence group is 
\begin{itemize}\itemsep=0ex
\item
a point transformation in the associated extended space of independent and dependent variables, 
involved derivatives (more precisely, the corresponding jet variables) and arbitrary elements,~and 
\item
projectable to the underlying space of variables, 
i.e., its components associated with independent and dependent variables do not depend on arbitrary elements. 
\end{itemize}
These conditions of locality and projectability 
can be weakened in the course of generalizing the definition of usual equivalence group
either singly or jointly~\cite{Meleshko1994,Popovych&Kuzinger&Eshraghi2010}.
At the same time, we need to preserve the principal features 
allowing one to refer to a transformation as equivalence transformation. 
These features should definitely include the consistency with the contact structure of the underlying jet space 
and the preservation of the class under consideration. 
Moreover, in order to be treatable within the framework of the local approach, 
the transformation should at least be point with respect to the independent and the dependent variables 
while the arbitrary elements are fixed.
This is why the procedure of weakening the locality property is much more delicate 
than the straightforward neglect of the independence of certain transformation components on arbitrary elements.
Introducing nonlocalities with respect to arbitrary elements is realized 
via defining a covering for the auxiliary system of constraints for the arbitrary elements.

Considering classes of linear ODEs, 
we should simultaneously weaken the conditions of locality and projectability, 
which leads to the notion of \emph{generalized extended equivalence group}. 
The attributes ``extended'' and ``generalized'' are related to weakening locality and projectability, respectively.

As coefficients of linear ODEs depend only on~$t$, 
the formal definition of the class~$\mathcal L$ includes the auxiliary system
\begin{gather}\label{EqAuxiliarySystemForLinearODEs}
\frac{\p a_{i-1}}{\p x^{(m)}}=0,\quad \frac{\p b}{\p x^{(m)}}=0
\end{gather}
for the arbitrary elements $a_0$,~\dots, $a_{r-1}$, $b$.  

\medskip\par\noindent{\bf Notation.}
Here and in what follows 
the index~$m$ runs from 0 to~$r$, 
the indices~$i$, $j$, $k$ and~$l$ run from 1 to~$r$,  
i.e., $m=0,\dots,r$ and $i,j,k,l=1,\dots,r$, 
and we assume summation with respect to the repeated indices~$j$, $k$ and~$l$. 
\par\medskip

We extend the set of arbitrary elements by $r$ more functions~$\chi_1$,~\dots, $\chi_r$ of~$t$ 
and impose the condition that, for each equation~$\mathcal E$ from the class~$\mathcal L$, 
the associated values of the additional arbitrary elements constitute 
a fundamental set of solutions of the homogeneous equation corresponding to~$\mathcal E$. 
In other words, we construct the following covering of the auxiliary system~\eqref{EqAuxiliarySystemForLinearODEs}:
\begin{subequations}\label{EqAuxiliarySystemForReparamLinearODEs}
\begin{gather}
\frac{\p a_{i-1}}{\p x^{(m)}}=0,
\label{EqAuxiliarySystemForReparamLinearODEsA}\\[.3ex]
\frac{\p b}{\p x^{(m)}}=0,\quad \frac{\p\chi_i}{\p x^{(m)}}=0,\quad
\label{EqAuxiliarySystemForReparamLinearODEsB}\\[.3ex]
\frac{\p^r\chi_i}{\p t^r}+a_{r-1}\frac{\p^{r-1}\chi_i}{\p t^{r-1}}+\dots+a_1\frac{\p\chi_i}{\p t}+a_0\chi_i=0,
\label{EqAuxiliarySystemForReparamLinearODEsC}\\[.3ex]
W(\chi_1,\dots,\chi_r):=\det\left(\frac{\p^{i-1}\chi_{j}}{\p t^{i-1}}\right)\ne0,
\label{EqAuxiliarySystemForReparamLinearODEsD}
\end{gather}
\end{subequations}
i.e., $W(\chi_1,\dots,\chi_r)$ denotes the Wronskian of~$\chi_1$,~\dots, $\chi_r$ with respect to the variable~$t$.
In view of the fact that each homogeneous linear ODE defines its fundamental set of solutions up to 
nonsingular linear combining, 
the penalty for extending the set of arbitrary elements is 
the appearance of gauge-equivalent tuples of arbitrary elements, 
i.e., the one-to-one correspondence between equations and tuples of arbitrary elements is lost.
The associated gauge equivalence group consists of the transformations of the form 
\begin{gather*}
\tilde t=t,\quad 
\tilde x^{(m)}=x^{(m)},\quad 
\tilde a_{i-1}=a_{i-1},\quad  
\tilde b=b,\quad  
\tilde \chi_i=\mu_{ij}\chi_j,  
\end{gather*}
where $\mu_{ij}$ are arbitrary constants with $\det(\mu_{ij})\ne0$.
See~\cite[Section~3.3.5]{LisleDissertation}%
\footnote{
Gauge equivalence transformations are called trivial in~\cite{LisleDissertation}.
}
and~\cite[Section~2.5]{Popovych&Kuzinger&Eshraghi2010} for a discussion of gauge equivalence. 

It is more convenient to interpret the introduction of the arbitrary elements~$\chi_1$,~\dots, $\chi_r$ 
as reparameterization of the class~$\mathcal L$, rather than the extension of its set of arbitrary elements. 
An equation of the form~\eqref{ODE} with a fundamental solution set $\{\chi_1,\dots,\chi_r\}$ 
can be represented as 
\[
\frac{W(\chi_1,\dots,\chi_r,x)}{W(\chi_1,\dots,\chi_r)}=b.
\]
This means that the arbitrary elements~$a_0$,~\dots, $a_{r-1}$ are completely defined by~$\chi_1$,~\dots, $\chi_r$, 
\[
a_{i-1}=\frac{(-1)^{r+i-1}}{W(\chi_1,\dots,\chi_r)}
\det\left(\frac{\p^m\chi_{j}}{\p t^m}\right)_{m\ne i-1},
\]
and hence the tuple of arbitrary elements reduces to $(\chi_1,\dots,\chi_r,b)$. 
The auxiliary system of constraints for the reduced tuple of arbitrary elements 
consists of the equations~\eqref{EqAuxiliarySystemForReparamLinearODEsB} 
and the inequality \eqref{EqAuxiliarySystemForReparamLinearODEsD}.

The reparameterized counterparts of subclasses of the class~$\mathcal L$ with $r\geqslant2$ 
are singled out from the reparameterization of~$\mathcal L$ 
by more constraining the arbitrary elements~$\chi_1$,~\dots, $\chi_r$, $b$.
For each subclass studied in Section~\ref{Equivalence_groupoids} we present 
the additional constraints to the auxiliary system~\eqref{EqAuxiliarySystemForReparamLinearODEsB}, \eqref{EqAuxiliarySystemForReparamLinearODEsD}
for the arbitrary elements of its reparameterization. 
Using the knowledge of the corresponding equivalence groupoid in the case $r\geqslant3$ 
and following the proof of Proposition~\ref{proposition1} in the case~$r=2$,  
we simultaneously give the general form of transformations from the generalized equivalence group of the reparameterized subclass.%
\footnote{%
We allow the dependence of transformation components on derivatives of arbitrary elements. 
In order to accurately interpret such transformations as elements of generalized equivalence groups, 
we should in fact extend the tuple of arbitrary elements and the auxiliary system for them 
by assuming all the derivatives of~$\chi_i$ up to order~$r$ as arbitrary elements and 
by considering relations between successive derivatives as constraints for these new arbitrary elements. 
} 
For such transformations we need to give all their components, including those associated with the arbitrary elements.

\medskip\par\noindent
$\mathcal L\colon\quad -\,$;
\begin{gather*}
\tilde t=T(t), \quad 
\tilde x=X_1(t)\big(x+\tilde X_0(t)\big),\quad
\tilde \chi_i=X_1(t)\mu_{ij}\chi_j,
\\[.5ex] 
\tilde b=\frac{X_1(t)}{(T_t(t))^r}\left(b+\frac{W(\chi_1,\dots,\chi_r,\tilde X_0(t))}{W(\chi_1,\dots,\chi_r)}\right),
\end{gather*}
where $T$, $X_1$ and $\tilde X_0$ are arbitrary smooth functions of~$t$ with $T_tX_1\ne0$ and 
$\mu_{ij}$ are arbitrary constants with $\det(\mu_{ij})\ne0$.

\medskip\par\noindent
$\widehat{\mathcal L}\colon\quad b=0$;
\begin{gather*}
\tilde t=T(t), \quad 
\tilde x=X_1(t)\big(x+\nu_j\chi_j\big),\quad
\tilde \chi_i=X_1(t)\mu_{ij}\chi_j,
\end{gather*}
where $T$ and $X_1$ are arbitrary smooth functions of~$t$ with $T_tX_1\ne0$ and 
$\mu_{ij}$ and~$\nu_j$ are arbitrary constants with $\det(\mu_{ij})\ne0$.

\pagebreak

\medskip\par\noindent
$\mathcal L_1\colon\quad \det\left(\dfrac{\p^m\chi_{j}}{\p t^m}\right)_{m\ne r-1}=0$;
\begin{gather*}
\tilde t=T(t), \quad 
\tilde x=C(T_t(t))^{\frac{r-1}2}\big(x+\tilde X_0(t)\big),\quad
\tilde \chi_i=C(T_t(t))^{\frac{r-1}2}\mu_{ij}\chi_j,
\\[.5ex] 
\tilde b=\frac C{(T_t(t))^{\frac{r+1}2}}\left(b+\frac{W(\chi_1,\dots,\chi_r,\tilde X_0(t))}{W(\chi_1,\dots,\chi_r)}\right),
\end{gather*}
where $T$ and $\tilde X_0$ are arbitrary smooth functions of~$t$ with $T_t\ne0$ and 
$C$ and~$\mu_{ij}$ are arbitrary constants with $C\det(\mu_{ij})\ne0$. 
Here and in the next case the powers of~$T_t$ for even~$r$ are interpreted in the same way as in footnote~\ref{FootnoteOnPowersOfTt}.

\medskip\par\noindent
$\widehat{\mathcal L}_1\colon\quad \det\left(\dfrac{\p^m\chi_{j}}{\p t^m}\right)_{m\ne r-1}=0,\quad b=0$;
\begin{gather*}
\tilde t=T(t), \quad 
\tilde x=C(T_t(t))^{\frac{r-1}2}\big(x+\nu_j\chi_j\big),\quad
\tilde \chi_i=C(T_t(t))^{\frac{r-1}2}\mu_{ij}\chi_j,
\end{gather*}
where $T$ is an arbitrary smooth function of~$t$ with $T_t\ne0$ and 
$C$, $\mu_{ij}$ and~$\nu_j$ are arbitrary constants with $C\det(\mu_{ij})\ne0$.

\medskip\par\noindent
$\mathcal L_2\colon\quad \det\left(\dfrac{\p^m\chi_{j}}{\p t^m}\right)_{m\ne r-1}=\det\left(\dfrac{\p^m\chi_{j}}{\p t^m}\right)_{m\ne r-2}=0$;
\begin{gather*}
\tilde t=\frac{\alpha t+\beta}{\gamma t+\delta}, \quad 
\tilde x=\frac C{(\gamma t+\delta)^{r-1}}\big(x+\tilde X_0(t)\big),\quad
\tilde \chi_i=\frac C{(\gamma t+\delta)^{r-1}}\mu_{ij}\chi_j,
\\[.5ex] 
\tilde b=C\frac {(\gamma t+\delta)^{r+1}}{(\alpha\delta-\beta\gamma)^r}
\left(b+\frac{W(\chi_1,\dots,\chi_r,\tilde X_0(t))}{W(\chi_1,\dots,\chi_r)}\right),
\end{gather*}
where $\alpha$, $\beta$, $\gamma$, $\delta$ and~$C$ are arbitrary constants with $\alpha\delta-\beta\gamma\ne0$ and $C\ne0$
that are defined up to obvious rescaling (so, only four constants among them are essential), 
$X_0$ is an arbitrary smooth function of~$t$
and $\mu_{ij}$ are arbitrary constants with $\det(\mu_{ij})\ne0$.

\medskip\par\noindent
$\widehat{\mathcal L}_2,\ r\geqslant3\colon\quad \det\left(\dfrac{\p^m\chi_{j}}{\p t^m}\right)_{m\ne r-1}=\det\left(\dfrac{\p^m\chi_{j}}{\p t^m}\right)_{m\ne r-2}=0,\quad b=0$;
\begin{gather*}
\tilde t=\frac{\alpha t+\beta}{\gamma t+\delta}, \quad 
\tilde x=\frac C{(\gamma t+\delta)^{r-1}}\big(x+\nu_j\chi_j\big),\quad
\tilde \chi_i=\frac C{(\gamma t+\delta)^{r-1}}\mu_{ij}\chi_j,
\end{gather*}
where $\alpha$, $\beta$, $\gamma$, $\delta$ and~$C$ are arbitrary constants with $\alpha\delta-\beta\gamma\ne0$ and $C\ne0$
that are defined up to obvious rescaling (so, only four constants among them are essential)
and $\mu_{ij}$ and~$\nu_j$ are arbitrary constants with $\det(\mu_{ij})\ne0$.

Up to gauge equivalence, the reparameterized subclass~$\widehat{\mathcal L}_2$ with $r=2$ 
consists of the single equation $x''=0$. 
Hence the generalized equivalence group of this subclass coincides with its usual equivalence group 
and is generated by the gauge equivalence group of this subclass and the point symmetry group of the equation $x''=0$. 
Recall that the last group consists of fractional linear transformations in the space of~$(t,x)$.
It is obvious that this subclass is normalized in the usual sense. 
Fixing the values $\chi_1=1$ and $\chi_2=t$ for canonical representatives of sets of gauge-equivalent tuples of arbitrary elements
leads to disappearing gauge equivalence transformations.

\medskip\par\noindent
$\mathcal A_1\colon\quad \chi_1=1$;
\begin{gather*}
\tilde t=T(t), \quad 
\tilde x=\frac{x+\tilde X_0(t)}{\mu_{1k}\chi_k},\quad
\tilde \chi_i=\frac{\mu_{ij}\chi_j}{\mu_{1k}\chi_k},
\quad 
\tilde b=\frac{(T_t(t))^{-r}}{\mu_{1k}\chi_k}\left(b+\frac{W(\chi_1,\dots,\chi_r,\tilde X_0(t))}{W(\chi_1,\dots,\chi_r)}\right),
\end{gather*}
where $T$ and $\tilde X_0$ are arbitrary smooth functions of~$t$ with $T_t\ne0$ and 
$\mu_{ij}$ are arbitrary constants with $\det(\mu_{ij})\ne0$.

\pagebreak 

\medskip\par\noindent
$\widehat{\mathcal A}_1\colon\quad \chi_1=1,\quad b=0$;
\begin{gather*}
\tilde t=T(t), \quad 
\tilde x=\frac{x+\nu_j\chi_j}{\mu_{1k}\chi_k},\quad
\tilde \chi_i=\frac{\mu_{ij}\chi_j}{\mu_{1k}\chi_k},
\end{gather*}
where $T$ is an arbitrary smooth function of~$t$ with $T_t\ne0$ and 
$\mu_{ij}$ and $\nu_j$ are arbitrary constants with $\det(\mu_{ij})\ne0$.

\medskip\par\noindent
$\mathcal A_2\colon\quad \chi_1=1,\quad \chi_2=t$;
\begin{gather*}
\tilde t=\frac{\mu_{2j}\chi_j}{\mu_{1k}\chi_k}, \quad 
\tilde x=\frac{x+\tilde X_0(t)}{\mu_{1k}\chi_k},\quad
\tilde \chi_i=\frac{\mu_{ij}\chi_j}{\mu_{1k}\chi_k},
\\[.5ex] 
\tilde b=\frac{(\mu_{1k}\chi_k)^{2r-1}}{\big(\mu_{2j}\mu_{1l}(\chi_j'\chi_l-\chi_j\chi_l')\big)^r}
\left(b+\frac{W(\chi_1,\dots,\chi_r,\tilde X_0(t))}{W(\chi_1,\dots,\chi_r)}\right),
\end{gather*}
where $\mu_{ij}$ are arbitrary constants with $\det(\mu_{ij})\ne0$
and $\tilde X_0$ is an arbitrary smooth function of~$t$.

\medskip\par\noindent
$\widehat{\mathcal A}_2,\ r\geqslant3\colon\quad \chi_1=1,\quad \chi_2=t,\quad b=0$;
\begin{gather*}
\tilde t=\frac{\mu_{2j}\chi_j}{\mu_{1k}\chi_k}, \quad 
\tilde x=\frac{x+\nu_j\chi_j}{\mu_{1k}\chi_k},\quad
\tilde \chi_i=\frac{\mu_{ij}\chi_j}{\mu_{1k}\chi_k},
\end{gather*}
where $\mu_{ij}$ and $\nu_j$ are arbitrary constants with $\det(\mu_{ij})\ne0$. 

For $r=2$, the reparameterized subclass~$\widehat{\mathcal A}_2$   
coincides with the reparameterized subclass~$\widehat{\mathcal L}_2$ 
since it consists of the single equation $x''=0$; see above.

\medskip

For each of the reparameterized classes~$\mathcal L$, $\mathcal L_1$ and~$\mathcal L_2$, 
its generalized equivalence group is not truly generalized 
since the transformation components for variables~$t$ and~$x$ do in fact not depend on arbitrary elements. 
Thus, this group coincides with the corresponding usual equivalence group, 
which generates the entire equivalence groupoid of the reparameterized class if $r\geqslant3$. 
In other words, the reparameterization preserves 
the normalization of the classes~$\mathcal L$, $\mathcal L_1$ and~$\mathcal L_2$ in the usual sense. 

At the same time, for the others of the above subclasses 
the presented equivalence groups are truly generalized 
and, if $r\geqslant3$, generate the entire equivalence groupoids of these subclasses. 
Therefore, although the classes~$\widehat{\mathcal L}$, $\widehat{\mathcal L}_1$, $\widehat{\mathcal L}_2$, 
$\mathcal A_1$, $\widehat{\mathcal A}_1$, $\mathcal A_2$ and $\widehat{\mathcal A}_2$ 
are not normalized in the usual or the generalized sense 
and, moreover, the classes associated with the Arnold forms are even not semi-normalized for $r\geqslant3$, 
the reparameterized versions of all these classes 
are normalized (resp.\ semi-normalized) in the generalized sense if $r\geqslant3$ (resp.\ $r=2$).  
It can be said that the reparameterization based on fundamental sets of solutions improves 
normalization properties of these classes. 

The arbitrary elements~$\chi_1$,~\dots, $\chi_r$ are nonlocally related to~$a_0$,~\dots, $a_{r-1}$. 
The generalized equivalence groups of the reparameterized counterparts generate the entire corresponding groupoids. 
Hence they are maximal, in certain sense, among generalized equivalence groups of classes 
obtained from the initial classes with replacing the corresponding auxiliary systems by their coverings.
Therefore, these groups can be considered as 
the \emph{generalized extended equivalence groups} of the subclasses of the class~$\mathcal L$. 
Summing up the above consideration, we obtain the following assertion. 

\begin{proposition}\label{ProposionOnGenExtNormalizationOfClassesOfLinearODEs}
The classes~$\widehat{\mathcal L}$, $\widehat{\mathcal L}_1$, $\widehat{\mathcal L}_2$, 
$\mathcal A_1$, $\widehat{\mathcal A}_1$, $\mathcal A_2$ and $\widehat{\mathcal A}_2$ with $r\geqslant3$
are normalized in the generalized extended sense. 
The corresponding generalized extended equivalence groups are related to 
the reparameterization based on fundamental sets of solutions. 
\end{proposition}

\pagebreak

\section{Conclusion}

In this paper we exhaustively describe the equivalence groupoid of the class~$\mathcal L$ of $r$th order linear ODEs
as well as equivalence groupoids of its subclasses~$\mathcal L_1$, $\mathcal L_2$, $\mathcal A_1$ and~$\mathcal A_2$
associated with the rational, the Laguerre--Forsyth, the first and second Arnold forms,
i.e., the classes of equations
of the form~\eqref{ODE}, \eqref{ODE1}, \eqref{ODE2}, \eqref{1st_Arnold's_form} and \eqref{2st_Arnold's_form}, respectively.
The corresponding classes~\smash{$\widehat{\mathcal L}$}, \smash{$\widehat{\mathcal L}_1$}, 
\smash{$\widehat{\mathcal L}_2$}, \smash{$\widehat{\mathcal A}_1$} and~\smash{$\widehat{\mathcal A}_2$}
of homogeneous equations are also studied from the point of view of admissible transformations.

The case $r=2$ is singular for all the above classes.
Each of the classes~$\mathcal L$, \smash{$\widehat{\mathcal L}$, $\mathcal L_1$, $\widehat{\mathcal L}_1$,}
$\mathcal L_2=\mathcal A_2$, $\mathcal A_1$ and~\smash{$\widehat{\mathcal A}_1$} with $r=2$
is an orbit of the free particle equation $x''=0$ with respect to the equivalence group of this class.
This is why its equivalence groupoid is generated by the compositions of transformations from its equivalence group with
transformations from the point symmetry group of the equation $x''=0$.
Hence the class is semi-normalized but not normalized, cf.\ the proof of Proposition~\ref{proposition2}.
The class \smash{$\widehat{\mathcal L}_2=\widehat{\mathcal A}_2$} is constituted by the single equation $x''=0$
and, thus, is normalized.

For $r\geqslant3$, the equivalence groupoids of the classes~$\mathcal L$, $\mathcal L_1$ and~$\mathcal L_2$
are generated by the corresponding (usual) equivalence groups.
In other words, each of these classes is normalized, see Propositions~\ref{proposition3}, \ref{proposition4} and \ref{proposition5}.
The associated subclasses of homogeneous equations are uniformly semi-normalized
with respect to the linear superposition symmetry groups of their equations.
This allows us to classify Lie symmetries of linear ODEs using the algebraic tools in three different ways.
The purpose of the presentation of various ways for carrying out the known classification
is to demonstrate advantages and disadvantages of each of them,
which is important, e.g., to effectively apply the algebraic approach to group classification of systems of linear ODEs.
Thus, the classification based on the Laguerre--Forsyth form, which is associated with a maximal gauge of arbitrary elements,
is just reduced to the classification of subalgebras of the algebra~$\mathfrak{sl}(2,\mathbb F)$,
which is finite dimensional (more precisely, three-dimensional).
The use of the rational form leads to involving the classification of all possible realizations
of finite-dimensional Lie algebras on the line.
At the same time, the single classification case of constant-coefficient equations in the rational form
is split in the Laguerre--Forsyth form into three (resp.\ two) cases over the real (resp.\ complex) field,
and two (resp.\ one) of them are related to variable-coefficient equations.
If we neglect the possibility of gauging arbitrary elements and consider general linear ODEs,
we need to classify specific realizations of specific Lie algebras in the space of two variables.

The structure of the equivalence groupoids of the classes~$\mathcal A_1$, \smash{$\widehat{\mathcal A}_1$,} $\mathcal A_2$ and~\smash{$\widehat{\mathcal A}_2$}
associated with the Arnold forms, where $r\geqslant3$, is more complicated
since these classes are even not semi-normalized.
This is why they are not usable for the group classification of the class~$\mathcal L$, 
although these are the forms that are involved in reduction of order of linear ODEs. 

Normalization properties of those of the above classes that are not normalized can be improved 
by reparameterizing these classes,  
cf.\ Proposition~\ref{ProposionOnGenExtNormalizationOfClassesOfLinearODEs}. 
At the same time, the reparameterization is not applicable to group classification of linear ODEs due to 
the complicated relation between old and new arbitrary elements,
the complex involvement of arbitrary elements in the new representation of equations 
and the appearance of gauge equivalence transformations. 

In contrast to single linear ODEs,
results concerning group properties of normal systems of second-order linear ODEs are very far from to be completed, not to mention general systems of linear ODEs;
see a more detailed discussion in~\cite{Boyko2013}.
Only recently the group classification of systems of second-order linear ODEs with commuting constant-coefficient matrices
was considered for various particular cases of the number of equations (two, three or four) and of the structure of the coefficient matrices
in a series of papers \cite{Campoamor2011,Campoamor2012,Meleshko2011,WafoSoh2010}
and was then exhaustively solved in~\cite{Boyko2013}.
In spite of a number of publications on the subject,
the group classification of systems of linear second-order ODEs with noncommuting constant-coefficient matrices or with general nonconstant coefficients
was carried out only for the cases of two and three equations~\cite{Moyo&Meleshko&Oguis2013,Suksern&Moyo&Meleshko2013,WafoSoh2000}.
The consideration of a greater number of equations or equations of higher and different orders
within the framework of  the standard ``compatibility'' approach requires cumbersome computations.

Even the least upper bound for the dimensions of the Lie symmetry algebras 
of normal systems of $r$th order ODEs in $n$ dependent variables is still not found for general values of $n$ and~$r>3$;
see the discussion of related results in \cite[p.~206]{Olver1995}. 
The best from known upper bounds when $r>3$ is $n^2+(r+1)n+2$~\cite{Gonzalez-Gascon&Gonzalez-Lopez1985}, 
but this is greater than the dimension of the Lie symmetry algebra of the elementary system $x_i{}^{(r)}=0$, $i=1,\dots,n$,  
which is equal to $n^2+rn+3$~\cite{Gonzalez-Gascon&Gonzalez-Lopez1983} 
and which may hypothetically coincide with the least upper bound. 
(The Lie symmetry algebras of all the elementary systems were computed in~\cite{Gonzalez-Gascon&Gonzalez-Lopez1983}.)
The least upper bound for $r=2$ is $n^2+4n+3$. 
It was first presented in \cite[pp.~68--69, Theorem~44]{Markus1959} but the proof contained a number of weaknesses although they may be eliminated.   
Later the least upper bound for $r=2$ was accurately derived in~\cite{Gonzalez-Gascon&Gonzalez-Lopez1983}. 
It is really minimal since it coincides with the dimension of the well-known Lie symmetry algebra of the free particle system $x_i''=0$, $i=1,\dots,n$, 
which is isomorphic to~$\mathfrak{sl}(n+2,\mathbb F)$~\cite{Gonzalez-Gascon&Gonzalez-Lopez1983}.
The least upper bound for $r=3$, which is $n^2+3n+3$, was found in~\cite{Fels1993}. 
It also coincides with the dimension of the Lie symmetry algebra of the corresponding elementary system $x_i'''=0$, $i=1,\dots,n$.
Moreover, it was shown in~\cite{Fels1993} using the Cartan equivalence method (see also~\cite{Fels1995}) that 
the Lie symmetry algebra of a system of second (resp.\ third) order ODEs is of maximal dimension if and only if 
this system is reduced by a point transformation to the relevant elementary system.
The particular case of this assertion for linear systems of second-order ODEs has earlier been proved in~\cite{Gonzalez-Lopez1988}.

Hence there is a demand for the development of new, more powerful, algebraic and geometric tools for the study of Lie symmetries,
which, for instance, involve a deep investigation of associated equivalence groupoids and other related algebraic structures.

\bigskip\par\noindent{\bf Acknowledgments.}
The research of ROP was supported by the Austrian Science Fund (FWF), project P25064.
VMB and ROP thank the University of Cyprus for the hospitality and financial support.

\end{document}